\documentclass[10pt]{amsart}
\usepackage{amsmath}
\usepackage{amsthm}
\usepackage{amsfonts}
\usepackage{amssymb}
\usepackage{amscd}
\usepackage[all]{xy}

\usepackage{hyperref}
\hypersetup{
    colorlinks=true,       
    linkcolor=blue,          
    citecolor=blue,        
    filecolor=blue,      
    urlcolor=blue           
}

    \newtheorem{Lem}{Lemma}[section]
    \newtheorem{Lem-Def}{Lemma-Definition}[section]
    \newtheorem{Prop}[Lem]{Proposition}
       
      \newtheorem*{thm1}{Theorem 1}      
      \newtheorem*{thm2}{Theorem 2}      
        \newtheorem*{Quest}{Question}
    
    \newtheorem{Thm}[Lem]{Theorem}  
    \newtheorem{Cor}[Lem]{Corollary}

\theoremstyle{definition}

    \newtheorem{Rem}[Lem]{Remark}

\newcommand{\Sym}{ \text{Sym}^{24}}

\newcommand{\co}{\mathbb{C}}\newcommand{\W}{\mathcal W}

\newcommand{\F}{\mathcal F}
\newcommand{\Z}{\mathcal Z}
\newcommand{\I}{\mathcal I}
\newcommand{\M}{\mathcal M}

\renewcommand{\L}{\mathcal L}
\renewcommand{\O}{\mathcal O}
\newcommand{\C}{\mathcal C}
\newcommand{\V}{\mathcal V}
\newcommand{\X}{\mathcal X}

\newcommand{\col}{\colon}

\newcommand{\Ps}{\mathbb{P}}
\newcommand{\ra}{\rightarrow}
\newcommand{\ol}{\overline}

\newcommand{\ze}{\mathbb{Z}}

\renewcommand{\:}{\colon}

\address{Marco Pacini, Universidade Federal Fluminense, Rua M. S. Braga, Niter\'oi (RJ) Brazil}
\email{pacini@impa.br}

\address{Damiano Testa, Mathematics Institute, University of Warwick, Coventry, CV4 7AL, United Kingdom}
\email{adomani@gmail.com}

\begin{document}

\title[Recovering plane curves of low degree]
{Recovering plane curves of low degree  from their inflection lines and inflection points}

\author[Marco Pacini and Damiano Testa]{Marco Pacini and Damiano Testa}

\begin{abstract}
In this paper we consider  the following problem: is it possible to recover a smooth plane curve of degree $d\ge 3$ from its inflection lines? We answer positively to the posed question for a general smooth plane quartic curve, making the additional assumption that also one inflection point is given, and for any smooth plane cubic curve.
\end{abstract}

\thanks{The first author was partially supported by CNPq, processo 300714/2010-6.}

\thanks{The second author was partially supported by EPSRC grant number EP/F060661/1.}

\maketitle

\section{Introduction}

In the last decade, several reconstruction theorems for plane and canonical curves defined over the field of complex numbers appeared in the literature. We mention the papers \cite{CS1} and \cite{L1}, the first one showing that a general smooth plane quartic curve can be recovered from its 28 bitangents and the second one generalizing the result to any smooth plane quartic. 
 The result is extended in \cite{CS2} to a general smooth canonical curve and in \cite{GS} to a general principally polarized abelian variety  considering theta hyperplanes,  the natural generalization of the bitangents of a plane quartic. Recently, an effective algorithm reconstructing  a canonical curve from its theta hyperplanes has been obtained  in \cite{L2}.

In this paper we consider other distinguished lines associated to a plane curve. Recall that a smooth plane curve of degree $d\ge 3$ has $3d(d-2)$ inflection lines, counted with multiplicity, i.e. lines cutting the curve in a point with multiplicity at least 3, called inflection point of the curve.  It is therefore natural to ask whether or not a plane curve of degree $d$ can be reconstructed from its inflections lines.   Here we  investigate the case of curves of  degree at most 4, i.e.\  plane cubic curves and plane quartic curves.

The properties of the inflections points of plane curves have been widely studied  by the classical geometers. For example, in \cite{H1,H2},  O.\ Hesse studied a pencil,  also known as the \emph{Hesse pencil},  given by the equation    
 $$\lambda_0 (x^3+y^3+z^3) - 3 \lambda_1 xyz = 0$$
 where $[\lambda_0,\lambda_1]\in \Ps^1$ and $x,y,z$ are homogeneous coordinates on $\Ps^2$. 
  The curves appearing in this pencil  share the same inflection points (see also the proof of Theorem~\ref{tu}).  For more details on the Hesse pencil, we refer the reader to \cite{AD}. In~\cite{W}, R.M.\ Winger considered curves of degree six sharing  some inflection lines. More precisely, he introduced  a pencil of plane curves of degree six with 12 common inflection points and 6 common inflection lines.  But, to our knowledge,  the problem of whether or not two distinct plane curves of degree $d$ could also share the whole set of inflection lines has never been considered, even for $d=3$,  and therefore this question is still open.

The interest in reconstructing results for a smooth plane quartic curve consists also in the attempt to give an improved version of the classical Torelli Theorem for non-hyperelliptic curves of genus 3. Indeed, let $C$ be a smooth plane quartic curve and  $J(C)$ its Jacobian, with principal polarization $\Theta(C)$. Let  
$\Theta(C)\ra (\Ps^2)^\vee$ be the Gauss map.  If $C$ can be recovered by a finite set of its dual curve, then only a finite number of points in the image of  
 the Gauss map are necessary to recover $C$. 
We refer to \cite[page~245]{ACGH} for more details on the Torelli map.

 For plane quartic curves, we need to make an additional assumption in the reconstruction result of Theorem \ref{main}, which we can state as follows.
 
\begin{thm1}\label{th1}
The general smooth plane quartic curve defined over the field of complex numbers is uniquely determined by its inflection lines and one inflection point.  In other words,  let $L\subset\Ps^2$ be a line and  $p\in L$ a point. Let  $X_1$ and $X_2$ be smooth plane quartic curves admitting $L$ as inflection line and $p$ as inflection point.  If $X_1$ is sufficiently general, and if $X_1$ and $X_2$ have the same inflection lines, then $X_1=X_2$. 
\end{thm1}

To prove Theorem \ref{th1}, we use the same degeneration technique  developed in  \cite{CS1}.  Indeed, in Section \ref{infle-lim}, we study degenerations of inflection lines when a smooth quartic approach a singular one and in Lemma \ref{nodal-cubic} we show that a nodal curve which is the union of a line and  an irreducible  cubic with a node is determined by limit inflection lines. Then we deduce the result for general smooth curves in Theorem \ref{main},  showing  that a certain morphism is \'etale exactly as in \cite{CS1}. In Lemma \ref{nodal-cubic} we prove that the map associating to a smooth plane quartic its configuration of inflection lines is generically finite onto its image.
Nevertheless,  the following problem remains open

\begin{Quest}
Is it possible to recover a smooth plane quartic curve  only from its inflection lines?
\end{Quest}

For plane cubic curves, we can state our reconstruction result contained in Theorem \ref{n3}.

\begin{thm2}\label{th2}
Let $C \subset \mathbb{P}^2$ be a smooth plane cubic curve over a field $k$ of characteristic different from three and let $T_C \subset (\mathbb{P}^2)^\vee$ be the set of inflection lines of $C$.  There is a unique (geometrically) integral curve $C' \subset (\mathbb{P}^2)^\vee$ such that 
\begin{itemize}
\item
if ${\rm char}(k) \neq 2$, then $C'$ is a sextic with cusps at the points of $T_C$;
\item
if ${\rm char}(k) =2$ and $j(C) \neq 0$, then $C'$ is a cubic containing $T_C$;
\item
if ${\rm char}(k) =2$ and $j(C) = 0$, then $C'$ is a cubic containing $T_C$, with vanishing $j$-invariant.
\end{itemize}
Moreover, the space of cubics in $(\mathbb{P}^2)^\vee$ containing $T_C$ has dimension one if and only if $({\rm char}(k),j(C)) = (2,0)$.  In all cases, the curve $C'$ described above is the dual of the curve $C$.
\end{thm2}

The proof of this result proceeds by reconstructing the dual of the initial plane cubic.  There are two configuration of nine points that are used in the 
argument: the nine inflection points in the plane contaning the cubic and the nine inflection lines in the dual projective plane.  The 
inflection lines are the base locus of a Halphen pencil of plane sextics of genus one.  If the characteristic of the ground field is different from two, then 
we prove that an integral plane sextics with a singularity at one of these points is automatically singular at all the nine points and the singularities are 
all cuspidal.  This cuspidal curve is the dual of the initial plane cubic and the reconstruction follows by projective duality.  In this proof it is important 
that the dual of the plane cubic is a curve of degree six with cuspidal singularities at nine points, and can therefore be identified with a curve in the 
anti-bicanonical linear system of a blow up of the dual projective plane.

\subsection{Notation}\label{Not}

Throughout the paper we  use the following notation and terminology. 
 In Section \ref{infle-lim} and Section \ref{rec-sec}, we work over $\co$. 
A \emph{curve} is a projective, local complete intersection connected and reduced scheme of pure dimension 1. If $C$ is a curve, then $g_C:=1-\chi(\O_C)$ is the \emph{genus} of $C$ and $\omega_C$ is its dualizing sheaf.  Moreover, $C^{sm}$ is the smooth locus of $C$ and $C^{sing}:=C-C^{sm}$. 
If $D$ is a  divisor of a curve $C$, we denote by $\nu_p D$ the multiplicity of $D$ at $p$, for a point 
$p\in C^{sm}$.

A \emph{nodal curve} is a  curve whose singular points are nodes.  A \emph{subcurve} $Z$ of a curve $C$ is a non-empty union of irreducible components of $C$ such that $Z\subsetneq C$. If $Z_1$ and $Z_2$ are 
subcurves of a curve $C$ with no common components and such that 
$Z_1\cap Z_2$ are nodes of $C$, we denote by $\Delta_{Z_1\cap Z_2}$  the Cartier divisor of $Z_1$ or $Z_2$ defined as
 $\Delta_{Z_1\cap Z_2}:=\sum_{p\in Z_1\cap Z_2} p$.  If $Z$ is a subcurve of a curve $C$, we let $Z':=\overline{C-Z}$, and if  $Z\cap Z'$ are nodes of $C$, we let $\Delta_Z:=\Delta_{Z\cap Z'}$. We say that a curve $C$ defined over a field $k$  has  a \emph{cusp} (respectively a \emph{tacnode}) in a point $s\in C$ if the completion of the local ring of $C$ at $s$ is isomorphic to the quotient of the formal power series ring $k[[x,y]]$ by the ideal $(y^2-x^3)$ (respectively $(y^2-x^4)$).

A \emph{family of curves} is a proper and flat morphism $f\col\C\ra B$ whose fibers are curves. 
We denote by $\omega_f$ the relative dualizing sheaf of the family and $C_b:=f^{-1}(b)$, for $b\in B$. A \emph{smoothing} of a curve $C$ is a family $f\col\C\ra B,$ where   $B$ is the spectrum of a discrete valuation ring with closed point $0$ such that the general fiber is smooth and $C_0=f^{-1}(0)=C$.  
A \emph{regular smoothing} of a curve $C$ is a smoothing $\C\ra B$ of $C$ such that $\C$ is smooth
 everywhere except possibly at the points of $C$ which lie on exactly one irreducible component of $C$.   If $\L$ is a line bundle  on a curve $C$, we set $\deg_Z \L:=\deg \L|_Z$, for every subcurve $Z$ of $C$.   
 
  If $C$ is a plane curve, a \emph{smoothing of $C$ to plane curves} is a smoothing $f\: \C\ra B$, where 
  $\C\subset B\times\Ps^2$ and $f$ is the restriction of the projection onto the first factor. 
 The dual curve $C^\vee$ of $C$ is defined as the closure in $({\Ps^2})^\vee$ of the set representing the tangents at the smooth points of $C$.  It is worth to observe that if $C$ is singular, then
  $C^\vee$ is not the flat limit of the duals of smooth curves approaching $C$ (see \cite{K2} for more information on the flat limit of dual curves).  If $F(x,y,z)$ is a homogeneous polynomial in $x,y,z$, then $Z(F)\subseteq \Ps^2$ will denote its zero set.
 
 If $V$ is a variety and $p$ is a point of $V$, we denote by $T_p V$ the tangent space to $V$ at $p$.   
 If $S$ is a scheme, we denote by $S^{red}$ the reduced scheme associated to $S$.

\section{Limits of inflection lines}\label{infle-lim}

Let $C$ be a plane quartic curve.  If $C$ is irreducible, the Pl\"{u}cker formulas for plane curves provide the number of smooth inflection points of $C$.   On the other hand, D.\ Eisenbud and J.\ Harris in~\cite{EH1} consider the problem of determining what are the limits of Weierstrass points in families of curves degenerating to stable curves. 
 So far,  there are answers  for curves of compact type, in~\cite{EH1},  and for curves with at most two components, in~\cite{La, WL, EM}.  But  how do the \emph{inflection lines} degenerate in a smoothing of $C$? For example,   how many inflection lines degenerate to a line of the tangent cone at a node of $C$, or to a line contained in $C$? In \cite{K1}, the author addresses the problem for pencils of type 
 $\{G^n+tF\}_{t\in\mathbb{A}^1}$, where $n\ge 2$ and $F$ and $G$ are homogeneous polynomials with   $\deg F=n\deg G$.  
In this section, we will give the complete list of  limit inflection lines and their multiplicities for certain plane quartic curves.

\smallskip

Let $f\col\C\ra B$ be a regular smoothing of a curve $C$ whose irreducible components intersect at nodes (notice that $C$ need not be a nodal curve in the sense of Subsection \ref{Not}, but it a nodal curve  in the sense of \cite{EM}).  Let $Z_1,\dots,Z_m$ be the irreducible components of $C$ and assume that $g_C>0$.  Following \cite[Definition 2.6]{EM},  let $W^*$ be the Weierstrass subscheme of the generic fiber of $f$ and $W_f$  be \emph{the $f$-Weierstrass scheme}, i.e.  the closure of $W^*$ in $\C$. Consider the natural finite morphism $\psi_f\col W_f\ra B$. Set $W_{f,0}:=W_f\cap C$ and let $[W_{f,0}]$ be the associated cycle.

For every sheaf $\M=\omega_f(\sum_{j=1}^m t_j Z_j)$ and for each $i\in\{1,\dots,m\}$, define $\rho_{\M,i}$ to be the natural map
$$
\rho_{\M,i}\col f_*\M|_0\ra H^0(\M|_{Z_i}).
$$
By \cite[Theorem 2.2]{EM}, there is a unique $m$-tuple of sheaves $(\L_1,\dots,\L_m)$ on $\C$,  
where for each $i\in\{1,\dots,m\}$ we have $\L_i=\omega_f(\sum_{j=1}^m t_j Z_j)$, for $t_1,\ldots,t_m \in \ze$, and such that the following conditions hold: 
\begin{itemize}
\item[(i)]
the natural map $\rho_{\L_i,i}\col f_*\L_i|_0\ra H^0(\L_i|_{Z_i})$ is injective;
\item[(ii)]
the natural map $\rho_{\L_i,j}\col f_*\L_i|_0\ra H^0(\L_i|_{Z_j})$ is not zero, if $j\ne i$.
\end{itemize}

Fix $i \in \{1,\ldots, m\}$. 
The sheaf $\L_i$ defined above is called the \emph{canonical sheaf of $f$ with focus on $Z_i$}. 
Notice that a sheaf $\M=\omega_f(\sum_{j=1}^m t_j Z_j)$, for $t_j\in\ze$,  is the canonical sheaf of $f$ with focus on $Z_i$ if, whenever $j\ne i$,  the following conditions hold
\begin{equation}\label{focus-cond}
h^0(\M|_C)=g, \quad h^0(\M|_{Z'_i}(-\Delta_{Z'_i}))=0 \quad \text{and} \quad h^0(\M|_{Z'_j}(-\Delta_{Z'_j}))<g .
\end{equation}
Indeed, for every $j\in\{1,\dots,m\}$, consider the natural exact sequence 
\begin{equation}\label{ker}
0\ra\M|_{Z'_j}(-\Delta_{Z'_j})\ra\M|_C\ra \M|_{Z_j}\ra 0
\end{equation}
In general, $f_*\M|_0$ is a subspace of dimension $g$ of $H^0(\M|_C)$. 
If the conditions in (\ref{focus-cond}) hold, then  $f_*\M|_0=H^0(\M|_C)$,  
$Ker(\rho_{\M,i})=H^0(\M|_{Z'_i}(-\Delta_{Z'_i}))=0$ and $\dim(Ker(\rho_{\M,j}))=h^0(\M|_{Z'_j}(-\Delta_{Z'_j}))<g$ if $j\ne i$, from the long exact sequence in cohomology associated to (\ref{ker}). Thus, $\M$   is the canonical sheaf of $f$ with focus on $Z_i$.
  
If $\L_i$ is the canonical sheaf of $f$ with focus on $Z_i$, the \emph{limit canonical aspect of $f$ with focus on $Z_i$} is 
  $(Im(\rho_{\L_i,i}),\L_i|_{Z_i})$. 
  Let   $t_{i,1},\dots,t_{i,m}$ be the unique integers such that $t_{i,i}=0$ and $\L_i\simeq\omega_f(\sum_{j=1}^m t_{i,j} Z_j)$.  For every $p\in Z_i\cap Z_j$ such that $j\ne i$, we define the  \emph{correction number for $\L_i$ at $p$} as $c_{Z_i}(p):=t_{i,j}$.

If  $R_{Z_i}$ is the ramification divisor of the limit canonical aspect of $f$ with focus on $Z_i$, for every 
$i\in\{1,\dots,m\}$, then by \cite[Theorem 2.8]{EM} we have
\begin{equation}\label{cycle-form}
[W_{f,0}]=\sum_{i=1}^m R_{Z_i}+\sum_{p\in C^{\text{sing}}} c_p p
\end{equation}
where $c_p:=g(g-1-c_{Z_i}(p)-c_{Z_j}(p))$, for $p\in Z_i\cap Z_j$, $i\ne j$, $i,j\in\{1,\dots,m\}$. 

If $\L$ is a line bundle on a smooth curve of genus $g$ and $R_V$ is the ramification divisor of a linear system $V\subset H^0(\L)$ of dimension $r+1$,  recall that by the Pl\"{u}cker formula we have
\begin{equation}\label{plu}
\deg R_V=r(r+1)(g-1)+(r+1)\deg \L .
\end{equation}

\begin{Lem}\label{twist-canon}
 Let $f\:\C\ra B$ be a regular smoothing of a curve $C$ and 
  $Z$ be an irreducible component of $C$ such that $Z\cap Z'$ are nodes of $C$.  Then $h^0(\omega_f(Z)|_C)=g_C$. Moreover, if $\#\Delta_Z=1$ and  $g_Z\ge 1$, then $h^0(\omega_f(2Z)|_C)=g_C$. 
\end{Lem}

\begin{proof}
Fix $i\in \{1,2\}$. By Riemann-Roch,  $h^0(\omega_f(iZ)|_C)=g_C-1+h^0(\O_\C(-iZ)|_C)$.  Hence we only  need to   show that $h^0(\O_\C(-iZ)|_C)=1$. 
Consider the natural exact sequence 
$$0\ra\O_\C(-iZ)|_Z(-\Delta_Z)\ra \O_\C(-iZ)|_C\ra \O_\C(-iZ)|_{Z'}\ra 0.
$$
Now, $h^0( \O_\C(-iZ)|_{Z'})=h^0(\O_{Z'}(-i\Delta_{Z'}))=0$. Thus, from the long exact sequence in cohomology, $h^0( \O_\C(-iZ)|_C)=h^0( \O_\C(-iZ)|_Z(-\Delta_Z))$. If $i=1$, then  $h^0( \O_\C(-iZ)|_Z(-\Delta_Z))=h^0(\O_Z)=1$, because $Z$ is connected.  If $i=2$,  $\Delta_Z=\{p\}$ and  $g_Z\ge 1$, then  
$h^0( \O_\C(-iZ)|_Z(-\Delta_Z))=h^0(\O_Z(p))=1$.
\end{proof}

Although some parts of the following three lemmas are known to the specialists, we will provide complete proofs for the reader's convenience.

 \begin{Lem}\label{nod}
Let $C$ be an irreducible plane quartic with a node $p$ and $Y$ be the normalization of $C$ at $p$. Let $g\col\C\ra B$ be a smoothing of 
$C$ to plane quartics. Assume that $\C$ has a singularity of $A_1$-type at $p$. Consider the blowup  
$\pi\:\X\ra \C$  of $\C$ at $p$, and set $f:=g\circ \pi$ and $X:=f^{-1}(0)=Y\cup E$, where   
$E\simeq\Ps^1$ is the exceptional component of $\pi$.   
 If  $R_Y$ is the ramification divisor of the limit canonical aspect of $f$ with focus on $Y$, then  
$[W_{f,0}]=R_Y+ 3p_1+3 p_2$, where $\{p_1,p_2\}:=E\cap Y$
 and $\nu_{p_i} R_Y\in\{0,1\}$, for $i\in \{1,2\}$.
  \end{Lem}
 
 \begin{proof}
First of all, $\omega_f$ is the canonical sheaf of $f$ with focus on $Y$, because we have $h^0(\omega_f|_X)=3$, $h^0(\omega_f|_{Y'}(-\Delta_{Y'}))=h^0(\O_{\Ps^1}(-p_1-p_2))=0$, $h^0(\omega_f|_{E'}(-\Delta_{E'}))=h^0(\omega_Y)=2$, and hence  the conditions listed in (\ref{focus-cond}) hold.   
In particular, the correction number for $\omega_f$  at $p_i$ is $c_Y(p_i)=0$, for $i\in\{1,2\}$. 

 The morphism $\pi$ contracts $E$, hence $C=\pi|_Y(Y)$. Moreover, $\pi^*(\O_{\C}(1))\simeq\pi^*\omega_g\simeq\omega_f$, because  $g\:\C\ra B$ is a family of canonical curves. In this way,  $\pi|_Y$ is induced by a projective space of dimension two contained in the complete linear system $|\omega_f\otimes\O_Y|=|\omega_Y(p_1+p_2)|$. But $h^0(\omega_Y(p_1+p_2))=3$, by Riemann-Roch, hence $\pi|_Y$ is induced by 
 $|\omega_Y(p_1+p_2)|$.
 
Set $\L:=\omega_Y(-p_1-p_2)$. 
  We claim that $h^0(\L)=0$. Indeed,  by contradiction  assume that   $h^0(\L)\ne 0$. Then $\omega_Y\simeq\O_Y(p_1+p_2)$, because $Y$ is irreducible and $\L$ has degree zero. Thus,  $h^0(\O_Y(p_1+p_2))=2$ and $\omega_Y(p_1+p_2)\simeq\O_Y(p_1+p_2)^{\otimes 2}$. It follows that  $\pi|_Y$ factors via a degree 2 morphism $Y\ra \Ps^1$, and hence $C=\pi|_Y(Y)$ is not a plane quartic, which is a contradiction. 

If we set $\L_E:=\omega_f(Y)$, then the canonical sheaf of $f$ with focus on $E$ is $\L_E$, because  $h^0(\L_E|_X)=3$, by Lemma \ref{twist-canon},  
$h^0(\L_E|_{E'}(-\Delta_{E'}))=h^0(\omega_Y(-p_1-p_2))=0$, by the above claim, and $h^0(\L_E|_{Y'}(-\Delta_{Y'}))=h^0(\O_E)=1$.  In particular, the correction number for $\L_E$  at $p_i$ is $c_E(p_i)=1$, for $i\in\{1,2\}$.

Recall that $R_Y$ is the ramification divisor of $(V_Y,\omega_f|_Y)$, where $V_Y$ is the vector space $V_Y=\text{Im}(H^0(\omega_f|_X)\hookrightarrow H^0(\omega_f|_Y))$. 
We have 
$\omega_f|_Y\simeq \omega_Y(p_1+p_2)$ and $\dim V_Y=h^0(\omega_Y(p_1+p_2))=3$, hence $V_Y=H^0(\omega_Y(p_1+p_2))$. Fix $i\in\{1,2\}$.
 Notice that $V_Y(-p_i)=H^0(\omega_Y(p_{3-i}))=H^0(\omega_Y)\simeq \co^2$. It follows that 
  $\dim V_Y(-2p_i)=h^0(\omega_Y(-p_i))=1$, and hence 
 $\nu_{p_i}(R_Y)=\dim V_Y(-3p_i)\in \{0,1\}$. Moreover,  by ~\eqref{plu},  the ramification divisor of the limit canonical aspect of $f$ with focus on $E$ is $R_E=0$, hence  by ~\eqref{cycle-form} we get
$[W_{f,0}]=R_Y+3(2-c_E(p_1)) p_1+3(2-c_E(p_2))p_2=R_Y+3(p_1+p_2)$.
\end{proof}

\begin{Lem}\label{custac}
Let $C$ be an irreducible plane quartic with a cusp $p$ (respectively a tacnode $p$) and set $t:=8$ (respectively t:=12). Let $f\:\C\ra B$ be a smoothing of $C$ to plane quartics, with $\C$ smooth. 
Then there exists a  finite base change $B'\ra B$ totally ramified over  $0\in B$, such that 
we have $[W_{f',0}]=\sum_{i=1}^{24-t}q_i+tp$, where $p\not\in\{q_1,\dots,q_{24-t}\}$ and $f'\col \C\times_B B'\ra B'$ is the second projection morphism.
\end{Lem}

\begin{proof}
As explained in \cite[Section 3C]{HM}, there is a finite base change $B'\ra B$ totally ramified over $0\in B$ and a family of curves $g\col\X\ra B'$ such that, if  $f'\col \C':=\C\times_B B'\ra B'$ is the second projection morphism, then  $\X$ and $\C'$ are $B'$-isomorphic away from the central fibers and $X:=g^{-1}(0)=Y\cup E$, where $Y$ is the normalization of $C$ at $p$ and $E$ is a smooth connected curve of genus $g_E=1$  intersecting transversally $Y$ at the points over $p\in C$. In particular, $\#(Y\cap E)=1$, if $p$ is a cusp (respectively  $\#(Y\cap E)=2$, if $p$ is a tacnode).  Moreover, it follows from 
\cite[Lemma 5.1]{P} that, if $B'\ra B$ has degree 6 (respectively 4), then there exists  a family $g\col \X\ra B'$  as stated, with $\X$ is smooth along $Y\cap E$ and with  a $B'$-morphism $\pi\:\X\ra\C'$ which is an isomorphism away from the central fibers and
 contracting $E$ to $p$.
 In particular, it follows that $\pi|_Y(Y)=C$.
 We are done if we show that 
$[W_{g,0}]=\sum_{i=1}^{24-t} q_i+\sum_{i=1}^t q'_i$, 
where $\{q_1,\dots,q_{24-t}\}\subset Y-E$ and $\{q'_1,\dots,q'_t\}\subset E$.  Actually we will prove that
 $\{q'_1,\dots,q'_t\}\subset E-Y$.

Set $\L_Y:=\omega_g(E)$. We claim that $\L_Y$ is the canonical sheaf of $g$ with focus on $Y$. Indeed, $h^0(\L_Y|_X)=3$, by Lemma \ref{twist-canon}, $h^0(\L_Y|_{Y'}(-\Delta_{Y'}))=h^0(\O_E(-\Delta_E))=0$, $h^0(\L_Y|_{E'}(-\Delta_{E'}))=h^0(\omega_Y(\Delta_Y))=2$, by Riemann-Roch. In particular,  for every $r\in Y\cap E$, the correction number for $\L_Y$  at $r$ is $c_Y(r)=1$.

Assume that $p$ is a tacnode and set 
$\{r_1,r_2\}:=Y\cap E$.  Arguing as before, $\L_E:=\omega_g(Y)$ is the canonical sheaf of $g$ with focus on $E$.  In particular,  the correction number for $\L_E$  at $r_i$ is $c_E(r_i)=1$, for $i\in\{1,2\}$.   Recall that $R_Y$ is   the ramification divisor of  $(V_Y,\L_Y|_Y)$, where $V_Y$ is the vector space $V_Y=\text{Im}(\rho_Y\:H^0(\L_Y|_X)\hookrightarrow H^0(\L_Y|_Y))$, $\rho_Y(s)=s|_{Y}$.   
Assume that  $s'(r_i)=0$, for $i\in\{1,2\}$, where $s'=\rho_Y(s)\in V_Y$. Then  $s|_E=0$, because $\L_Y|_E\simeq\O_E$ and hence
\begin{equation}\label{VYbis}
V_Y(-r_i)\subseteq \rho_Y(H^0(\L_Y|_X(-r_1-r_2)))=H^0(\O_Y(r_1+r_2))=\co^2.
\end{equation}
The other inclusion in~\eqref{VYbis} is clear, hence 
$\dim V_Y(-r_i)=2$, 
$\dim V_Y(-2r_i)=h^0(\O_Y(r_{3-i}))=1$ and   
$\dim V_Y(-3r_i)=h^0(\O_Y(r_{3-i}-r_i))=0$, because $g_Y=1$. In this way, 
$\nu_{r_i}R_Y=0$, for $i\in\{1,2\}$.  By (\ref{plu}), we have $\deg R_Y=12$,  hence 
$R_Y=\sum_{i=1}^{12} q_i$, where $\{q_1,\dots,q_{12}\}\subset  Y-E$.   Arguing similarly, 
if $R_E$ is the ramification divisor of the limit canonical aspect of $g$ with focus on $E$,  we have $R_E=\sum_{i=1}^{12} q'_i$, where $\{q'_1,\dots,q'_{12}\}\subset E-Y$.   It follows from (\ref{cycle-form}) that  $[W_{g,0}]=R_Y+R_E$, and hence we are done if $p$ is a tacnode.

From now on, assume that $p$ is a cusp and set $r:=Y\cap E$.   Recall that  $C=\pi|_Y(Y)$. Moreover, we have $\pi^*(\O_{\C'}(1))\simeq\pi^*\omega_{f'}$, because  $f'\:\C'\ra B'$ is a family of canonical curves, hence $\pi^*(\O_{\C'}(1)) \simeq \omega_g\otimes\O_\X(D)|_X$, for some Cartier divisor $D$ of $\X$ of type $D=aY+bE$, for $a,b\in\ze$. But $\deg_{E} \pi^*(\O_{\C'}(1))=0$, because $\pi$ contracts $E$, and $\deg_E\omega_g=1$, hence necessarily $\O_\X(D)|_X\simeq\O_\X(E)|_X$.  In this way,  $\pi|_Y$ is induced by a projective space of dimension two contained in $|\omega_g(E)\otimes\O_Y|=|\omega_Y(2r)|$. But $h^0(\omega_Y(2r))=3$, by Riemann-Roch, hence $\pi|_Y$ is induced by 
 $|\omega_Y(2r)|$.
 
Set $\L:=\omega_Y(-2r)$.
 We claim that $h^0(\L)=0$. Indeed, suppose 
 by contradiction that  $h^0(\L)\ne 0$. 
Then $\omega_Y\simeq\O_Y(2r)$, because $Y$ is irreducible and $\L$ has degree zero,  implying 
$h^0(\O_Y(2r))=2$ and $\omega_Y(2r)\simeq\O_Y(2r)^{\otimes 2}$. It follows that  $\pi|_Y$ factors via a degree 2 morphism $Y\ra \Ps^1$, hence $C=\pi|_Y(Y)$ is not a plane quartic, which is a contradiction. 

Set $\L_E:=\omega_g(2Y)$. We claim that $\L_E$ is the canonical sheaf of $g$ with focus on $E$. Indeed, $h^0(\L_E|_X)=3$, by  Lemma \ref{twist-canon}, $h^0(\L_E|_{E'}(-\Delta_{E'}))=h^0(\omega_Y(-2r))=0$, by the above claim, 
 and  $h^0(\L_E|_{Y'}(-\Delta_{Y'}))=h^0(\O_E(2r))=2$, by Riemann-Roch.  
 In particular, the correction number for $\L_E$  at $r$ is $c_E(r)=2$. 
 
 Recall that $R_Y$ is the ramification divisor of $(V_Y,\L_Y|_Y)$, where $V_Y$ is the vector space
 $V_Y=\text{Im}(H^0(\L_Y|_X)\hookrightarrow 
 H^0(\L_Y|_Y))$. 
 Notice that $h^0(\L_Y|_{Y})=h^0(\omega_Y(2r))=3$, hence $V_Y=H^0(\omega_Y(2r))$. We have   
 $\dim V_Y(-r)=h^0(\omega_Y(r))=2$, $\dim V_Y(-2r)=h^0(\omega_Y)=2$, 
 $\dim V_Y(-3r)=h^0(\omega_Y(-r))=1$ and $\dim V_Y(-4r)=h^0(\omega_Y(-2r))=0$. In this way, $\nu_r R_Y=2$, and we can write $R_Y=2r+\sum_{i=1}^{\deg R_Y-2} q_i$, where  
$\{q_1,\dots,q_{\deg R_Y-2}\} \subset Y-E$.
 
 Recall that $R_E$ is the ramification of $(V_E,\L_E|_E)$, where $V_E$ is the vector space 
 $V_E=\text{Im}(H^0(\L_E|_X)\hookrightarrow H^0(\L_E|_E))$. Notice that  
 $h^0(\L_E|_{E})=H^0(\O_E(3r))=3$, hence $V_E=H^0(\O_E(3r))$. We have 
  $\dim V_E(-r)=h^0(\O_E(2r))=2$, $\dim V_E(-2r)=h^0(\O_E(r))=1$, $\dim V_E(-3r)=h^0(\O_E)=1$ and 
$\dim V_E(-4r)=h^0(\O_E(-r))=0$. In this way,  $\nu_r R_E=1$, and we can write $R_E=r+\sum_{i=1}^{\deg R_E-1} q'_i$, where $\{q'_1,\dots,q'_{\deg R_E-1}\}\subset E-Y$.  
 By (\ref{plu}), 
 $\deg R_E=9$, hence (\ref{cycle-form}) implies that 
 $[W_{g,0}]=R_Y+R_E-3r=\sum_{i=1}^{16}q_i+\sum_{i=1}^8 q'_i$, and we are done  if $p$ is a cusp.
\end{proof}

\begin{Lem}\label{degen} 
Let $C$ be a  nodal plane quartic curve which is the union of a smooth irreducible cubic $Y$ and a line $Z$, and set  $\{p_1,p_2,p_3\}:=Y\cap Z$.  If $f\:\C\ra B$ is a regular smoothing  of $C$, then 
$$[W_{f,0}]=R+(s_1+\dots+s_6)+3(p_1+p_2+p_3)$$
where $\{s_1,\dots,s_6\}\subset Z\cap C^{sm}$ and $R$ is the ramification divisor of $|\O_Y(1)|$.
 \end{Lem}

\begin{proof}
Set  $\L_Y:=\omega_f$.  The canonical sheaf of $f$ with focus on $Y$ is $\L_Y$, because  $h^0(\L_Y|_C)=3$, 
  $h^0(\L_Y|_{Y'}(-\Delta_{Y'}))=h^0(\omega_Z)=0$, and 
  $h^0(\L_Y|_{Z'}(-\Delta_{Z'}))=h^0(\omega_Y)=1$.  In particular, the correction number for $\L_Y$  at 
  $p_i$ is $c_Y(p_i)=0$, for $i\in\{1,2,3\}$.  Set $\L_Z:=\omega_f(Y)$. The canonical sheaf of $f$ with focus on $Z$ is $\L_Z$, because   $h^0(\L_Z|_C)=3$, by Lemma \ref{twist-canon}, $h^0(\L_Z|_{Z'}(-\Delta_{Z'}))=h^0(\omega_Y(-\Delta_Y))=0$ and 
  $h^0(\L_Z|_{Y'}(-\Delta_{Y'}))=h^0(\omega_Z(\Delta_Z))=2$.  
 In particular the correction number for $\L_Z$  at 
  $p_i$ is $c_Z(p_i)=1$, for $i\in\{1,2,3\}$.   If  $R_Z$ is   the ramification divisor of the limit canonical aspect of $f$ with focus on $Z$,  then 
$\deg(R_Z)=6$, by (\ref{plu}). Hence, by (\ref{cycle-form}),  we are done if we show that $\nu_{p_i}R_Z=0$, for $i\in\{1,2,3\}$.  
  
  Consider  $V_Z=\text{Im}(\rho_Z\:H^0(\L_Z|_C)\hookrightarrow H^0(\L_Z|_Z))$,  $\rho_Z(s)=s|_Z$. 
 Assume that  $s'(p_1)=0$, for some $s'=\rho_Z(s)\in V_Z$. Then  $s|_Y=0$, because $\L_Z|_Y\simeq\O_Y$, and   hence
\begin{equation}\label{VW}
V_Z(-p_1)\subseteq\rho_Z(H^0(\L_Z|_C(-p_1-p_2-p_3)))=H^0(\omega_Z(p_1+p_2+p_3))=\co^2.
\end{equation}
Since the other inclusion in~\eqref{VW} is clear, we get $\dim V_Z(-p_1)=2$, and hence also 
$\dim V_Z(-2p_1)=1$ and $\dim V_Z(-3p_1)=0$. In this way,   
  $\nu_{p_1}R_Z=0$, and similarly   $\nu_{p_2}R_Z=\nu_{p_3}R_Z=0$, and hence  we are done.
   \end{proof}

Let $\Ps^{14}$ be the projective space parameterizing plane quartic curves and $[C]\in \Ps^{14}$ be the point  parameterizing a plane quartic $C$.   Let $\V\subset \Ps^{14}$ be the open subset parameterizing reduced quartic curves  that are GIT-semistable with respect to the natural action of $PGL(3)$ on  
$\Ps^{14}$ and with finite stabilizer.   If  $[C]\in \V$, then $C$ is reduced and a singular point of $C$ is a node, a cusp or a tacnode. The double conics are the unique GIT-semistable non-reduced quartics and if $C$ is a smooth quartic, then $[C]\in \V$.   We refer  to \cite[Chapter 4.2]{MFK}, \cite[Section 3.4]{CS1} or \cite{AF} for a detailed list of the curves parameterized by $\V$.

If $C$ is a smooth plane curve, an \emph{inflection line} of $C$ is a line cutting $C$ in a point with multiplicity at least 3.  
Recall that a smooth plane quartic admits exactly 24 inflection lines, counted with multiplicity. We denote by    $\F_C\in \Sym({\Ps^2}^\vee)$  the configuration of  inflection lines of a smooth plane quartic curve $C$.   Let $\V^0\subset \V$ be the open subset  parameterizing  smooth plane quartics and consider 
   the morphism  $\F^*\: \V^0\ra \Sym({\Ps^2}^\vee)$ such that  $\F^*([C])=\F_C$, for every $[C]\in \V^0$.

\begin{Prop}\label{F-map} 
 There exists a morphism $\F\:\V\ra \Sym({\Ps^2}^\vee)$ such that $\F|_{\V^0}=\F^*$.
\end{Prop}

\begin{proof}
Let $\Gamma$ be the closure inside $\V\times \Sym({\Ps^2}^\vee)$ of the graph 
$$\Gamma^0:=\{([C],\F_C) : [C] \in \V^0\}\subset  \V\times \Sym({\Ps^2}^\vee).$$
 Let  $\pi_1\:\Gamma\ra \V$ and $\pi_2\: \Gamma\ra   \Sym({\Ps^2}^\vee)$ be respectively  the restriction of the first and the second  projection morphism.  
 
 If $[C]\in \V$, we claim that  only a finite number of lines can be limits of inflection lines of smooth curves degenerating to $C$. In fact, recall that $C$ is reduced and its singularities are double points. Consider a smoothing of  $C$, and let  $L$ be a line which is  a limit of inflection lines of the general curve of the family. It follows that $L$ cuts  $C$ with multiplicity at least 3 in a point $p$. If $L$ is a linear component of $C$, then it varies in a finite set.   
Suppose that $L\not\subset C$. If $p\in C^{sing}$, then $p$ is a double point of $C$, and $L$ is one of the lines of the tangent cone of $C$ at $p$. In particular, $L$ varies in a finite set. If $p\in C^{sm}$, then $L$ corresponds to a singular point of the dual curve $C^\vee$, hence $L$ varies again in a finite set, because we work over the field of complex numbers, and hence $C^\vee$ is reduced. 
 
 In particular, the fibers of $\pi_1$ are finite.  Notice that  $\V$ is smooth, because it is an open subset of $\Ps^{14}$. Moreover $\pi_1|_{\Gamma^0}$ is an isomorphism onto the smooth variety $\V^0$, thus
  $\Gamma^0$, and hence $\Gamma$, are irreducible.  Since  $\pi_1$ is a birational morphism,   $\pi_1$ is an isomorphism, by \cite[Corollaire 4.4.9]{G}.  The required morphism is  $\F:=\pi_2\circ\pi_1^{-1}$.
\end{proof}

\begin{Rem}
In the proof of Proposition \ref{F-map} we did not use the condition that the points of $\V$ parametrize GIT-semistable curve. Indeed, the same proof shows that $\F^*$ actually extends to the open subset of $\Ps^{14}$ parametrizing reduced quartics with singular points of multiplicity two.
\end{Rem}

For every $[C]\in\V$,  set  $\F_C:=\F([C])\in \Sym({\Ps^2}^\vee)$. We can view $\F_C$ as a (possibly non-reduced) hypersurface of $\Ps^2$. 
We call a line  $L\subset \Ps^2$  an \emph{inflection line of} $C$ if $L$ is a component of $\F_C^{red}$. 
If $L$ is not a component of $\F_C$, we set $\mu_L\F_C=0$, otherwise  
$\mu_L \F_C$ will be the multiplicity of $L$ as a component of $\F_C$.  
For every $L\subset\F_C^{red}$, we say that $L$ is \emph{of type 0} if $L\cap C\subset C^{sm}$, of  \emph{type 1} if  $L\cap C^{sing}\ne\emptyset$ and $L\not\subset C$, and \emph{degenerate} if $L\subset C$.

\begin{Rem}\label{imp-obs}
To compute the multiplicity of a component of $\F_C$, for $[C]\in \V$, we will use the following observation. Consider the incidence variety 
$$
\I:=\{([C],L): L\subseteq \F_C\}\subseteq \V\times(\Ps^2)^\vee
$$
and the finite morphism $\pi_\I\col\I\ra \V$ of degree 24 obtained by restricting the first projection $\V\times (\Ps^2)^\vee\ra \V$ to $\I$. The cycle associated to the fiber of $\pi_\I$ over $[C]$ is
$\sum_{L\subseteq \F^{red}_C} (\mu_L\F_C) L$.  Notice that if $\C\ra B$ is a smoothing of $C$ to general  smooth plane curves, with associated morphism $B\ra \V$, and if we set $\I_B:=B\times_\V \I$, then the first  projection $\pi_{\I_B}\col\I_B\ra B$ is a finite morphism of degree 24 and the cycle associated to the  fiber of $\pi_{\I_B}$ over $0\in B$ is again  $\sum_{L\subseteq \F^{red}_C} (\mu_L\F_C) L$.  We would like to warn the reader to a possibly confusing point: while the limits of the inflection {\em{lines}} do not depend on the chosen smoothing, the limits of the inflection {\em{points}} may depend on the smoothing.
\end{Rem}

\begin{Lem}\label{mult-norm}
Let $B$ be the spectrum of a discrete valuation ring with closed point $0$. 
Let $\gamma \col C\ra B$ be a finite morphism  and 
$\nu\col C^\nu\ra C$ be the normalization of $C$. Let $\sum_{p\in\gamma^{-1}(0)} c_p p$ and $\sum_{(\gamma\circ\nu)^{-1}(0)} d_q q$ be the cycles associated to the fiber  respectively of $\gamma$ and $\gamma\circ\nu$  over $0$, where $c_p$ and $d_q$ are positive integers. Then for every $p\in\gamma^{-1}(0)$ we have 
$c_p=\sum_{q\in \nu^{-1}(p)} d_q.$
\end{Lem}

\begin{proof}
If $p\in C^{sm}$, we have nothing to prove. If $p\in C^{sing}$, pick  the normalization $\alpha\col D\ra C$ of $C$ at $p$. Then $\nu=\alpha\circ\beta$, where $\beta\col C^\nu\ra D$ is the normalization of $D$. Let $\sum_{r\in (\gamma\circ\alpha)^{-1}(0)} e_r r$ 
be the cycle associated to the fiber of $\gamma\circ\alpha$ over $0$, where $e_r$ is a positive integer.  If $d$ is the degree of $\gamma$, then we have
$$c_p=d-\sum_{p\ne p'\in\gamma^{-1}(0)} c_{p'}=d-\underset{r\not\in\alpha^{-1}(p)}{\sum_{r\in (\gamma\circ\alpha)^{-1}(0)}}e_r=\sum_{r'\in \alpha^{-1}(p)} e_{r'}=\sum_{q\in\nu^{-1}(p)} d_q,$$
where the second equality follows because $\alpha$ is an isomorphism away from $\alpha^{-1}(p)$, and the fourth equality follows because $\beta$ is an isomorphism locally at any $q\in\beta^{-1}(\alpha^{-1}(p))=\nu^{-1}(p)$.
\end{proof}

\begin{Lem}\label{dual-cubic}
Let $f\col\C\ra B$  be a smoothing of  an irreducible plane cubic  $C$ with a node (respectively a cusp).  Set $k=3$ (respectively $k=1$).  Then exactly   $k$  lines cutting $C$ in a smooth point with multiplicity  $3$ are degenerations of  inflection lines of the general fiber of $f$.
\end{Lem}

\begin{proof}
Let $h$ be the number of lines cutting $C$ in a smooth point with multiplicity 3.  By \cite[pages 153-154]{EH2}, any such line is degeneration of an inflection line of the general fiber of $f$, hence we are done if we show that $h=k$.  The dual curve $C^\vee\subset({\Ps^2})^{\vee}$  of $C$ is  irreducible and, since we work over the complex numbers, its singular locus  consists  exactly of $h$ cusps. It follows that $g_{C^\vee}-h=0$, because $C$ and $C^\vee$ are birational. Let $d$ be the degree of  $C^\vee$. If $C$ has a node, then $d=4$ and hence $h=g_{C^\vee}=3$.  If $C$ has a cusp, then $d=3$ and hence $h=g_{C^\vee}=1$.
\end{proof}

\begin{Prop}\label{F-conf}
Let  $[C]\in\V$ and $L\subset \Ps^2$ a line.  The line $L$ is an inflection line of $C$ if and only if either $L$ is contained in the tangent cone of $C$ at some singular point of $C$, or $L$ is an inflection line of type 0 of $C$.  Moreover, assuming that $L$ is an inflection line of $C$, the multiplicity of $L$ in $\F_C$ satisfies the following conditions: 
\begin{itemize}
\item[(i)]
if $L$ is a line of the tangent cone of $C$ at some $p\in C^{sing}$ such that either $L\not\subset C$ or $L\subset C$ and $L$ is tangent  at $p$ to a component of $C$ different from $L$, then $\mu_L \F_C \in\{3,4\}$,  $\mu_L \F_C=8$, $\mu_L \F_C=12$ if   $p$ is  respectively  a node, a cusp, a tacnode; 
 \item[(ii)]
if $L$ is a component of $C$ which is not tangent to any other component of $C$, then $\mu_L\F_C=6$;
\item[(iii)]
if $L$ is an inflection line of type 0 of $C$, then $\mu_L \F_C\le 2$,  and if equality holds, then  $C$ is irreducible. 
\end{itemize}
\end{Prop}

\begin{proof}
Let $p$ be the limit of an inflection point in a fixed smoothing of $C$.  Through the point $p$ there is a line $L$ in $\F_C$ meeting $C$ with multiplicity 
at least three at $p$.  If $p\in C^{sm}$, then the line $L$ is an inflection line of type 0, or $L$ is a component of $C$.  If  $p\in C^{sing}$, then the line $L$ is contained in the tangent cone to $C$ at $p$, since the point $p$ has multiplicity 2 in $C$ and at least 3 in $L \cap C$.  This proves that the lines of $\F_C$ must be of the stated form.

We now prove the converse, establishing at the same time the results about the multiplicities.  
Let $g\:\C\ra B$ be  a smoothing  of $C=g^{-1}(0)$ to general plane quartics, 
and let  $b\col B\ra \V$ be the associated morphism. 
Recall the varieties $\I\hookrightarrow\V\times (\Ps^2)^\vee$ and 
$\I_B=B\times_\V \I\hookrightarrow B\times_\V\V\times (\Ps^2)^\vee\simeq B\times (\Ps^2)^\vee$ defined in Remark \ref{imp-obs}. 
 Consider the relative Gauss $B$-map $\theta\:\C\dashrightarrow  B\times (\Ps^2)^\vee$ defined as  
$\theta(q)=(g(q),T_q C_{g(q)})$, where $T_q C_{g(q)}$ is the tangent line to $C_{g(q)}=g^{-1}(g(q))$ at 
$q$.  Notice that   $\theta$ is defined away from $C^{sing}$. 
 Consider the $g$-Weierstrass scheme $W_g\hookrightarrow \C$ and let $\psi_g\: W_g\ra B$ be the associated finite morphism of degree 24.  Let  
 $\nu\:W_g^\nu\ra W_g$ be the normalization of $W_g$ and  
 $\gamma\: W_g^\nu\ra B\times(\Ps^2)^\vee$  the morphism extending $\theta|_{W_g}$.   Notice that
  $\theta$ sends the points of $W_g$ away from $\psi_g^{-1}(0)$ to points of $\I_B$, hence 
  $\gamma$ factors through the inclusion $\I_B\hookrightarrow B\times (\Ps^2)^\vee$. We get the following commutative diagram.
  \[
\SelectTips{cm}{11}
\begin{xy} <16pt,0pt>:
\xymatrix{
 \C \ar@{-->}[rr]^{\theta\,\,\,\,\,\,\,\,\,\,} & & B\times (\Ps^2)^\vee\UseTips\ar[r]  &\V\times (\Ps^2)^\vee  \\
 W_g\ar@{^{(}->}[u] \ar[d]^{\psi_g} & W_g^\nu \ar[l]_{\nu}  \UseTips\ar[r]^{\gamma}& \I_B \ar[r]\ar[d]^{\pi_{\I_B}}\ar@{^{(}->}[u]&  \I\ar@{^{(}->}[u] \ar[d]^{\pi_\I} \\
   B \ar@{=}[rr] &  &  B \ar[r]^{b}   &  \V \\
 }
\end{xy}
\]
   It follows from Remark~\ref{imp-obs} that $\gamma(\nu^{-1}(p))$ parametrizes  the set of lines of $\F_C$ cutting $C$ with multiplicity at least 3 in  $p$, for every 
 $p\in\psi_g^{-1}(0)$. In particular, 
 \begin{equation}\label{gamma-inj}
 \gamma(\nu^{-1}(p))\cap \gamma(\nu^{-1}(p'))=\emptyset,
 \end{equation}
  for every $p,p'\in\psi_g^{-1}(0)$ such that 
 $p\ne p'$ and such that the line through $p$ and $p'$ is not contained in $C$.

Assume that  $C$ is irreducible and that 
 $L$ is a line of the tangent cone of $C$ at some $p_0\in C^{sing}$, so that   $L\not\subset C$. Suppose that $p_0$ is a cusp or a tacnode.  Choose a smoothing $g\col \C\ra B$ of $C$ with $\C$ smooth. Then  
  there is a finite base change totally ramified over $0\in B$ such that the statement of Lemma \ref{custac} holds.  To avoid cumbersome notations, we keep the same symbols for the family obtained from 
  $g\col\C\ra B$ after the base change.   By Lemma \ref{custac}, we have $[W_{g,0}]=\sum q_i +tp_0$, where $p_0\not\in\{q_1,\dots,q_{24-t}\}$ and 
  $t=8$,  if $p_0$ is a cusp, and $t=12$,  if $p_0$ is a tacnode. 
 Observe that $p_0\in W_g$ and $L$ is the unique line cutting $C$ with multiplicity at least 3 at $p_0$, 
and hence $\gamma(\nu^{-1}(p_0))=(0,L)$. Using (\ref{gamma-inj}), we deduce the following set-theoretic equality  
\begin{equation}\label{p0-fib}
\nu^{-1}(p_0)=\gamma^{-1}((0,L)).
\end{equation}   Since  $\pi_{\I_B}\circ\gamma=\psi_g\circ\nu$, we have
  $$[(\pi_{\I_B}\circ\gamma)^{-1}(0)]=[(\psi_g\circ\nu)^{-1}(0)]=\underset{s\not\in \nu^{-1}(p_0)}{\sum_{s\in(\psi_g\circ\nu)^{-1}(0)}}u_s s+\sum_{s\in\nu^{-1}(p_0)}u_s s$$
  where $u_s$ are positive integers. Applying Lemma  \ref{mult-norm} to $\psi_g$ and $\psi_g\circ\nu$,
   it follows that $t=\sum_{s\in\nu^{-1}(p_0)}u_s$. 
 By  Remark~\ref{imp-obs},  we have $[\pi_{\I_B}^{-1}(0)]=\sum_{M\subseteq\F_C^{red}}(\mu_M \F_C) M$.
   On the other hand, $W_g$ and $\I_B$ are isomorphic over $B\setminus\{0\}$, hence  
     $\gamma\col W^\nu_g\ra \I_B$ is the normalization of  $\I_B$. Applying again Lemma \ref{mult-norm} to $\pi_{\I_B}$ and $\pi_{\I_B}\circ\gamma$, and using  (\ref{p0-fib}), it follows that 
     $$\mu_L \F_C=\sum_{s\in \gamma^{-1}((0,L))} u_s =\sum_{s\in\nu^{-1}(p_0)} u_s =t.$$
     We obtain that $L\in\F_C$ and that condition (i) holds.

Suppose that $p_0$ is a node of $C$.  Choose a smoothing $g\col\C\ra B$ of $C$ where $\C$ has a singularity of $A_1$-type at $p_0$.  Let
$\pi\col \X\ra \C$ be the blowup of $\C$ at $p_0$ and $f\col\X\ra B$ be the composed morphism $f:=g\circ\pi$. 
Write $f^{-1}(0)=Y\cup E$, where $Y$ is the normalization of $C$ at $p_0$ and $E$ is the exceptional component of $\pi$.  Consider the $f$-Weierstrass scheme $W_f\hookrightarrow \X$ and let $\psi_f\col W_f\ra B$ be the associated finite morphism of degree 24. 
 Since $W_f$ and $W_g$ are isomorphic over $B\setminus\{0\}$, it follows that $W^\nu_g$ is the normalization of $W_f$. 
 Let $\ol{\nu}\col W^\nu_g\ra W_f$ be the normalization map.
We get the following commutative diagram.
  \[
\SelectTips{cm}{11}
\begin{xy} <16pt,0pt>:
\xymatrix{
 \X\ar[r]^{\pi} & \C \ar@{-->}[r]^{\theta\,\,\,\,\,\,\,\,} & B \times(\Ps^2)^\vee \\
 W_f\ar@{^{(}->}[u] \ar[d]^{\psi_f} & W_g^\nu\ar[l]_{\ol\nu}  \UseTips\ar[r]^{\gamma}& \I_B \ar[d]^{\pi_{\I_B}}\ar@{^{(}->}[u]\\
   B \ar@{=}[rr] &  &  B     \\
 }
\end{xy}
\]
It follows from  Lemma  \ref{nod} that  
$[W_{f,0}]=\sum_{i=1}^{18-\epsilon_1-\epsilon_2} q_i+(3+\epsilon_1)p_1+(3+\epsilon_2)p_2$, for some $\epsilon_1,\epsilon_2\in\{0,1\}$, where 
$\{p_1,p_2\}:=E\cap Y$ and $\{q_1,\dots,q_{18-\epsilon_1-\epsilon_2}\}\subset Y-E$. 
The set  $\gamma(\ol\nu^{-1}(p_i))$ parametrizes the lines of $\F_C$ cutting $C$ with multiplicity at least 3 in $p_0$ and tangent to the branch of $p_0$ determined by $p_i$, for $i\in\{1,2\}$. In this way, we have
$\gamma(\ol\nu^{-1}(p_i))=(0,L_i)$,  where  $L_i$ is the unique line of the tangent cone of $C$ at $p_0$ which is tangent to the branch of $p_0$ determined by   $p_i$,  for $i\in\{1,2\}$. Arguing as in the case where $p_0$ is a cusp or a tacnode, we have that  $\mu_{L_i}\F_C=3+\epsilon_i$, for $i\in\{1,2\}$. 
  We obtain that $L_i\in\F_C$, for $i\in\{1,2\}$ and that condition (i) holds.

We complete the proof for $C$  irreducible as follows. 
Suppose that $C$ is irreducible and that $L$ is an inflection line $L$ of type 0 of $C$ . The non-reduced point $p_0$ of $L\cap C$ has multiplicity at most 4, hence $p_0$ has multiplicity at most 2 in $[W_{g,0}]$.  Arguing again as in the case where $p_0$ is a cusp or a tacnode, it follows that $\mu_L \F_C\le 2$.

 From now on, assume that $C$ is reducible.  
Suppose that $L$ is a line such that  $L\subset C$.   If $C$ is nodal and 
$\ol{C-L}$ is  an irreducible smooth plane cubic, pick   a regular smoothing $g\:\C\ra B$ of $C$ to general plane quartics, and consider the first diagram introduced in the proof.
Let $s_1,\dots,s_6$ be the points of Lemma \ref{degen}, so that     
$\{s_1,\dots,s_6\}\subset  L\cap C^{sm}\cap W_g$. Recall that $\gamma(\nu^{-1}(s_i))$ parametrizes the set of lines of $\F_C$ cutting $C$  with multiplicity at least 3 in $s_i$. Since 
any line different from $L$ cuts $C$  with multiplicity 1 in
$s_i$,    
necessarily  $\gamma(\nu^{-1}(s_i))=(0,L)$, and hence, arguing again as in the case where $p_0$ is a cusp or a tacnode, we have   $\mu_{L}\F_C\ge 6$. In the remaining cases,  if we take a deformation of $C$ to nodal
curves which are  the union of $L$ and a smooth cubic, we see that $\mu_{L}\F_C\ge 6$, because $L$ is constant along this deformation. 

Suppose that 
$L$ is a line of the tangent cone of $C$ at some $p_0\in C^{sing}$. Then $L$ is the specialization of inflection lines of type 1  of a family of curves whose general member is  irreducible with exactly one singular point of the same analytic type  of $p_0$. Hence   $\mu_{L}\F_C\ge t$, where $t\in\{3,4\}$, if $p_0$ is a node, $t=8$ if $p_0$ is a cusp, and $t=12$, if $p_0$ is a tacnode.  
To complete the proof, we only need to show that conditions (i), (ii) and (iii) hold for a reducible curve.
  In what follows, we will use the above inequalities and that $\sum_{L\subseteq\F_C}\mu_{L}\F_C=24$.  
  
Suppose that $C=Q_1\cup Q_2$,  for  smooth conics $Q_1,Q_2$. If $C$ is nodal and   $L_1,\dots,L_8$ are the lines of the tangent cone of $C$ at its nodes, then $\mu_{L_i}\F_C\ge 3$, and hence necessarily $\mu_{L_i}\F_C=3$, for $i\in\{1,\dots,8\}$.  
 If  $Q_1$ and $Q_2$ intersects each other transversally at 2 points and they are  tangent at a point $r$, necessarily 
  $\mu_{L_0}\F_C=12$ and $\mu_{L_i}\F_C=3$, for $i\in\{1,\dots,4\}$, where 
  $L_0$ is the tangent cone of $C$ at the tacnode $r$ and $L_1,\dots,L_4$ are the lines of the tangent cone of $C$ at its nodes.
  
 Suppose that  $C=Q\cup L_1\cup L_2$, for  a smooth conic $Q$ and lines $L_1,L_2$.  If $C$ is nodal,    necessarily $\mu_{L_1}\F_C=\mu_{L_2}\F_C=6$ and $\mu_{L_i}\F_C=3$, for $i\in\{3,\dots,6\}$, where   $L_3,\dots,L_6$ are  the tangents to $Q$ at the points of the set  $Q\cap (L_1\cup L_2)$.
 If $Q$ is tangent to $L_1$ and transverse to $L_2$,  
 necessarily $\mu_{L_1}\F_C=12$,  $\mu_{L_2}\F_C=6$ and $\mu_{L_3}\F_C=\mu_{L_4}\F_C=3$, where $L_3,L_4$ are the tangents to $Q$ at the points of the set $Q\cap L_2$. 
 
 If $C=\cup_{i=1}^4 L_i$, for   lines $L_1,\dots,L_4$,  necessarily  $\mu_{L_i}\F_C=6$,  for $i\in\{1,\dots,4\}$.

 Suppose that $C=Y\cup L$,  for an   irreducible plane cubic $Y$ and a line $L$. Set $k=3$ (respectively $k=1$) if $Y$ has  a node  (respectively a cusp).   Assume that $Y\cap L$ are nodes and set  
 $\{p_1,p_2,p_3\}:=Y\cap L$ . Let $L_i$ be the tangent to $Y$ at $p_i$ for $i\in\{1,2,3\}$.   
 We distinguish two cases. 
 
 In the first case, $Y\cap L$ contains no inflectionary point of $Y$. If  $Y$ is smooth, we can use  Lemma \ref{degen}.  Thus, arguing as for the inflection lines of type 0 when $C$ is irreducible, it follows that  there are $9$ inflection lines of type 0 of $C$, which are exactly the inflection lines of $Y$, and necessarily 
  $\mu_{L}\F_C=6$, $\mu_{L_i}\F_C=3$, for $i\in\{1,2,3\}$ and the $9$ inflection lines of type 0 of 
   $C$ have multiplicity 1. If $Y$ is singular,  pick  a deformation of $C$ to plane curves which are union of $L$ and a smooth plane cubic.    It follows from Lemma  \ref{dual-cubic} that  $C$ has $k$ inflection lines of type 0, and 
   hence necessarily   $\mu_{L}\F_C=6$,  $\mu_{L_i}\F_C=3$, for $i\in\{1,2,3\}$, $\mu_M\F_C=3$ (respectively  $\mu_M\F_C=8$), where $M$ is a line of the tangent cone at  a node (respectively a cusp) of $Y$, and the  $k$ inflection lines of type 0 of $C$ have multiplicity 1.  We isolate the case in which $C$ is nodal in Remark~\ref{ivv}.
 
  In the second case, there is a positive integer $k'\le \text{min}(k,3)$ such that $p_i$ is an inflectionary point of $Y$, 
   for $i\in\{1,\dots,k'\}$, where $k:=9$ if $Y$ is smooth.  Pick  a deformation of $C$  to plane curves which are union of $Y$ and a general line intersecting $Y$ transversally.  Then   exactly two inflection lines specialize to  $L_i$, for $i\in\{1,\dots,k'\}$, one of type 0  of multiplicity 1 and one of type 1 of multiplicity  3,  and $C$ has   $k-k'$ inflection lines of type 0.  Hence $\mu_{L}\F_C=6$, $\mu_{L_i}\F_C=4$ for $i\in\{1,\dots,k'\}$, $\mu_{L_i}=3$ for $i\in\{k'+1,\dots,3\}$,  $\mu_M\F_C=3$ (respectively  $\mu_M\F_C=8$), where $M$ is a line of the tangent cone at  a node (respectively a cusp) of $Y$, and the  $k-k'$ lines of type 0 of $C$ have multiplicity 1.

  Suppose that $C=Y\cup L$, for  an irreducible plane cubic $Y$  and a line $L$  such that $Y$ and $L$ are tangent at a point $t$. Notice that $t$  is a non-inflectionary point of  $Y$, because $[C]\in \V$.  Write $\{t,n\}:=Y\cap L$.  Consider a deformation of $C$ to plane curves which are union of $Y$ and a general line containing   $n$ and intersecting $Y$ transversally.  Then exactly 3 inflection lines of the general fiber of the deformation degenerate to $L$, of multiplicities respectively 6, 3, 3, hence $\mu_L\F_C=12$, while the other inflection lines are constant along the family.  
  
Notice that if $C$ is reducible, then the inflection lines of type 0 of $C$ have always multiplicity 1. We get that   conditions (i),  (ii) and (iii) hold. 
  \end{proof}
  
\begin{Rem} \label{ivv}
For ease of reference we observe the following immediate consequence of Proposition~\ref{F-conf}.
\begin{itemize}
\item 
If  either $C$ is irreducible, with $C^{sing}$ consisting exactly of one tacnode, or if $C$ is the union of a line $L$ and a smooth cubic $Y$ such that $L$ and $Y$  are tangent at a non-inflectionary   point of $Y$, then $\F_C$ contains at least 6 inflection lines of type 0 each one of which has multiplicity at most 2.
\item 
If $C$ is  nodal and $C=Y\cup L$, where $Y$ is  an irreducible cubic  with a node  and $L$ is a line such that $Y\cap L$ contains no inflectionary points of $Y$,  then $\F_C$ consists of the line $L$ with multiplicity 6, the three inflection lines at smooth points of $Y$ of type 0 with multiplicity 1, and five lines of type 1 of multiplicity 3, two coming from the inflection lines  through the node of $Y$, and three coming from inflection lines through the intersection $Y \cap L$ and different from $L$ (see also the proof of Proposition~\ref{F-conf}.)
\end{itemize}
\end{Rem}

\begin{Prop}\label{F-stab}
Consider  the natural action of $PGL(3)$ on $\Sym({\Ps^2}^\vee)$. Let $[C]\in\V$ and $\delta,\gamma,\tau$ be respectively  the number of nodes, cusps, tacnodes of $C$. Then the following conditions hold
\begin{itemize}
\item[(i)]
 if  $\F_C$ is GIT-unstable, then either $\gamma\ge 2$ or $\tau\ge 1$.
 \item[(ii)]
 if $\gamma=\tau=0$, then $\F_C$ is GIT-stable.
 \end{itemize}
 \end{Prop}
 
\begin{proof}
Throughout the proof,  we assume that $\gamma\le 1$ and $\tau=0$. 
For a point $p\in \Ps^2$, let $\mu_p\F_C:=\sum_{p\in L\subset\F_C}\mu_L\F_C$.  By~\cite[Proposition~4.3]{MFK} or~\cite[Theorem~4.17]{N}, $\F_C$ is GIT-semistable  if and only if  $\mu_p\F_C\le 16$, for every $p\in\Ps^2$,  and  $\mu_L \F_C\le 8$, for every $L\subset\F_C$, and GIT-stable if and only if both inequalities are strict.   By Proposition \ref{F-conf},   $\mu_L\F_C\le 8$, and the equality holds if and only if $\gamma=1$.  Thus, we only need to check that $\mu_p\F_C\le 16$, with $\mu_p\F_C<16$ if $\gamma=0$,  for every $p\in\Ps^2$. 

So, fix $p\in\Ps^2$.   Let $M^p_0$  be the set of inflection lines of type 0 of $C$ containing $p$. By Proposition \ref{F-conf}(iii), $\mu_L\F_C\le 2$, for every $L\in M_0^p$. Let $M^p_1$ be  the set of inflection lines of type 1 of $C$ containing $p$. We distinguish 3 cases. 

\smallskip 

(a) Suppose  that $C$ is irreducible.  If $p\in C^{sing}$, then $\mu_p\F_C\le 8$,  by Proposition \ref{F-conf}. Hence we can  assume  $p\not\in C^{sing}$. Let $\pi\:C^{\nu}\ra\Ps^1$ be the resolution of the projection  map  from $p$, where $C^{\nu}$ is the normalization of $C$, and $R$ be the ramification divisor of  $\pi$.  

First of all, assume that $p\not\in C$.  Then $\#\{q\in C^{\nu}: \nu_q R\ge 2\}=\#(M^p_0\cup M^p_1)$ and 
 $\deg R=12-2\delta-2\gamma$, by Riemann-Hurwitz. Assume that $\gamma=0$ and, by contradiction, that  $\mu_p\F_C\ge 16$.  Observe that $\#M^p_1\le 3$,  otherwise $p$ would be a node, which is a contradiction. 
   If $\#M^p_1=0$,  then $\#M^p_0\ge 8$,  hence $\deg R\ge 16$.   If $\#M^p_1=1$, 
    then $\#M^p_0\ge 6$, hence   $\deg R\ge 14$. 
If $\#M^p_1=2$, then  $\delta\ge 2$, otherwise $p$ would be a node,  and $\#M^p_0\ge 4$ which implies $\deg R\ge 12$.  If $\#M^p_1=3$,  then $\delta\ge 3$, otherwise $p$ would be a node,  and  
$\#M^p_0\ge 2$, implying  $\deg R\ge 10$. In any case, we get a contradiction. 
Assume now that $\gamma=1$ and, by contradiction, that $\mu_p\F_C>16$.   Let $L$ be the line of the tangent cone at the cusp of $C$, with  $\mu_{L}\F_C=8$.  If $p\notin L$, then $\mu_p\F_C\le 16$,  a contradiction.   
Hence $p\in L$ and $\#M^p_1\ge 1$.  If $\#M^p_1=1$, then $\#M^p_0\ge 5$, hence $\deg R\ge 12$.
  If $\#M^p_1=2$, then $\delta\ge 1$ and $\#M^p_0\ge 3$, hence $\deg R \ge 10$.
   If $\#M^p_1=3$, then $\delta\ge 2$ and   $\#M^p_0\ge 1$, hence $\deg R \ge 8$. In any case, we get a contradiction. 

Suppose that $p\in C^{sm}$.   Then $\nu_p R\le 2$ and, if  $T_p C$ is the tangent to $C$ at $p$, we have $\#\{q\in C^{\nu}: \nu_q R\ge 2\}\ge\#((M^p_0\cup M^p_1)-\{T_p C\})$.  Notice that $\deg R=10-2\delta-2\gamma$, by Riemann-Hurwitz, and we can conclude arguing exactly as in the case $p\not\in C$.

\smallskip 

(b) Suppose that  $C=Q_1\cup Q_2$,  for conics $Q_1,Q_2$.   Hence $C$ is nodal, because $\tau=0$.   If $Q_1,Q_2$ are irreducible, then $p$ is contained  in at most four inflection lines of $C$, and hence  $\mu_p\F_C\le 12$.   If $Q_1$ is irreducible and $Q_2$ is union of two lines $L_1,L_2$,  then $\mu_p\F_C\le 12$, where the equality holds if and only if either $p\in L_1\cap L_2$ or $p\in L_i\cap L\cap L'$, for $i\in\{1,2\}$,  where $L,L'$ are components of $\F_C$ such that $\{L,L'\}\cap\{L_1,L_2\}=\emptyset$. 
If $Q_1$ and $Q_2$ are reducible, i.e. if $C$ is the union of four lines $L_1,\dots,L_4$, 
then  $\mu_p\F_C\le 12$, where the equality holds if and only if $p\in L_i\cap L_j$, for $i,j\in\{1,\dots,4\}$ and $i\ne j$. 

\smallskip 

(c)  Suppose that  $C=Y\cup L$, for   an irreducible cubic $Y$  and a line $L$. By contradiction, assume $\mu_p\F_C\ge 16$.  By Proposition \ref{F-conf}, $\mu_{L}\F_C=6$  and  $\mu_{L'}\F_C=1$, for  $L'\in M^p_0$.   We have  
$\#M^p_1\le 2$ and if $p\in C^{sing}$, then 
   $\mu_p\F_C\le 10$,  a contradiction.   
If $p\in L-Y$ and $Y$ is smooth, then  $\#M^p_0\ge 10$, a contradiction because $\#M^p_0\le 9$, by Proposition \ref{F-conf}.  If $p\in L-Y$ and $Y$ has a node,  then $\#M^p_1\le 1$ and  $\#M^p_0\ge 6$, a contradiction because $\#M^p_0\le 3$, by Proposition \ref{F-conf}.
   If $p\in L-Y$ and $Y$ has a cusp, then $\#M^p_1\le 1$ and $\#M^p_0\ge 2$, a contradiction because $\#M^p_0\le 1$, by Proposition~\ref{F-conf}.

   So  we can assume  $p\not\in C^{sing}\cup L$.   If $\mu_{L'}\F_C\ge 3$, for some $L'\subset\F_C$, then $p\in L'$, otherwise 
   $\mu_p\F_C\le 15$, a contradiction.  
 In this way, if $Y$ is smooth, then  $\#M^p_1=3$ and hence also
   $\#M^p_0\ge 4$.  We get a contradiction considering the projection  
$Y\ra\Ps^1$ from $p$ and by Riemann-Hurwitz.  
 If $Y$ has a node $q$,  then  $p$ is contained in the two lines of the tangent cone at $q$, hence $p=q\in C^{sing}$,  a contradiction. 
If $Y$ has a cusp $q$ and $L_0$ is the line of the tangent cone at $q$,  then  $\#M^p_1=4$,  with 
$L_0\in M^p_1$. 
 We get a contradiction by considering the resolution  
$Y^\nu\ra\Ps^1$  of the projection from $p$, where $Y^\nu$ is the normalization of $Y$, and by Riemann-Hurwitz.   
\end{proof}

\section{Recovering quartics from inflection lines and inflection points}\label{rec-sec}

Fix a line $L\subset \Ps^2$ and a point $p\in L$.  Recall that $\V\subset\Ps^{14}$ is the variety parametrizing reduced plane quartic curves which are GIT-semistable with respect to the natural action of $PGL(3)$ and with finite stabilizers. Let $\V_L\subset \V$ be the locus defined as the closure in $\V$ of 
 \begin{equation}\label{VL}
\V^0_L\:=\{[C]\in \V : L\not\subset C \text{ and } L\subset \F_C^{red}\}.
 \end{equation}
Let  $\V_{L,p}\subset \V_L$ be the locus defined as the closure in $\V_L$ of 
\begin{equation}\label{VLp}
\V^0_{L,p}:=\{[C]\in \V^0_L: p \text{ is the non-reduced point of } C\cap L\}.
\end{equation}

\begin{Lem}\label{Tang}   
Let  $[C]\in \V$ with $C$ nodal and $L\subset \Ps^2$ be a line. Then  
  \begin{itemize}
  \item[(i)]
  if $[C]\in \V^0_L$ and $C\cap L=\{p,q\}\subset C^{sm}$,  where  $p$ is  the non-reduced point of $C\cap L$, and $p\ne q$,  then 
   $$T_{[C]} \V_L\simeq H^0(C,\O_C(4)\otimes\O_C(-2p)).$$
  \item[(ii)]
    if $[C]\in \V^0_L$ and  $C\cap L=\{p,q\}$, with $p$  a node of $C$ and $q\in C^{sm}$,  and if  $\nu\: C^\nu\ra C$ is the normalization of $C$ at $p$, with  $\{p_1,p_2\}:=\nu^{-1}(p)$, where  $L$ is the tangent to the branch of $p$ determined by $p_1$,  then
 $$T_{[C]} \V_L\simeq H^0(C^\nu,\nu^*\O_C(4)\otimes \O_{C^\nu}(-2p_1-p_2)).$$    
  \item[(iii)]
  if   $L\subset C$ and $p\in C^{sm}\cap L$, then  $[C]\in \V_{L,p}$ and 
 $$T_{[C]} \V_{L,p}\simeq H^0(C,\O_C(4)\otimes \O_C(-3p)).$$
  \end{itemize}
\end{Lem}

\begin{proof}
 Let $x,y,z$ be  homogeneous coordinate of $\Ps^2$. Let $a_0,\dots,a_{14}$ be the homogeneous  coordinates of the projective space $\Ps^{14}$ parametrizing plane quartic curves.  Let $\co[x,y,z]_d$ be the vector space of homogeneous polynomials of degree $d$ in $x,y,z$.
 
 Suppose that $C$ is as in (i) or (ii).  
First of all, we will show that $\V^0_L$ is smooth of dimension 12 at $[C]$.  Up to an isomorphism 
of $\Ps^2$,  we can assume that $L=Z(z)$.  Notice that  $[Z(F)]\in \V^0_L$  for some
$F\in \co[x,y,z]_4$ if and only if   $F(x,y,0)$ is a binary quartic
$$F(x,y,0)=a_0x^4+a_1x^3y+a_2x^2y^2+a_3xy^3+a_4y^4$$
where $[a_0,a_1,a_2,a_3,a_4]\in \Ps^4_{a_0,\dots,a_4}$ and 
the binary quartic has at least a triple root.  
Then we have a rational map 
$\tau \colon \V^0_L\dashrightarrow L\times L$ defined as follows: if $[C'] \in \V^0_L$ and 
$C' \cap L = \{p',q'\}$ with $p' \neq q'$ and $p'$ is the non-reduced point of $C' \cap L$, then 
we let $\tau([C']) = (p',q')$.  
Near $[C]$, the map $\tau$ is defined by the assumption that $p$ and $q$ are different, its fibers are $\mathbb{A}^{10}_{a_5,\dots,a_{14}}$, and hence $\V^0_L$ is smooth of dimension 12 at $[C]$.  In particular,  $\dim(T_{[C]} \V_L)=12$.

Now, suppose that $C$ is as in (i).  By Riemann-Roch,  $h^0(C,\O_C(4)\otimes \O_C(-2p))=12$,  hence to conclude the proof of (i), we only need to show that the natural injective map $H^0(C,\O_C(4)\otimes \O_C(-2p))\hookrightarrow H^0(C,\O_C(4))$ factors via an injective map
\begin{equation}\label{incl1}
H^0(C,\O_C(4)\otimes \O_C(-2p))\hookrightarrow T_{[C]} \V_L.
\end{equation}

Indeed,  assume that $C=Z(F)$, for $F\in \co[x,y,z]_4$ and,  up to an isomorphism  of $L$,  that
  $F(x,y,0)=xy^3$, i.e. $p=[1,0,0]$ and $q=[0,1,0]$.  

Recall that  a  vector of $T_{[C]} \V$ is given by  a first order deformation  
$F+\epsilon\cdot  G$, where $G\in \co[x,y,z]_4$ 
 and  $\epsilon^2=0$.    This vector belongs to $T_{[C]} \V_L$ if and only if $(F+\epsilon\cdot G)(x,y,0)$ has at least a triple root.  
 If we view $G$ as a section of $\O_{\Ps^2}(4)$ and we consider  the surjective restriction map 
 $\rho\:H^0(\Ps^2,\O_{\Ps^2}(4))\ra H^0(C,\O_C(4))$, we obtain an identification  $T_{[C]} \V\simeq H^0(C,\O_C(4))$.  Notice that the vector space $H^0(C,\O_C(4)\otimes\O_C(-2p))$ is the image via $\rho$ of the set of sections of $\O_{\Ps^2}(4)$ induced by polynomials $\ol{G}\in\co[x,y,z]_4$ such that   $Z(\ol{G})$ and $C$, and hence $Z(\ol{G})$ and $L$, intersect each other with multiplicity at least 2 in $p$, i.e.  polynomials $\ol{G}$ such that  
\begin{equation}\label{Gbar}
\ol{G}(x,y,0)=y^2 H, 
\text{ for some } H\in \co[x,y]_2.
\end{equation}
  In this case,  write  $H=3 H' x+a y^2$, where    $H'\in \co[x,y]_1$ and $a\in\co$. Then 
$$(F+\epsilon \cdot \ol{G})(x,y,0)=F(x,y,0)+\epsilon\cdot \ol{G}(x,y,0)=xy^3+\epsilon\cdot y^2H=$$
$$=xy^3+\epsilon\cdot (a y^4+3 H' xy^2)=(y+\epsilon\cdot H')^3(x+\epsilon \cdot ay),$$
where in the last equation we used the relation $\epsilon^2=0$.  We see that $(F+\epsilon\cdot \ol{G})(x,y,0)$ has at least a triple root, hence it defines a vector of   $T_{[C]} \V_L$, and hence there exists an injective map as required in (\ref{incl1}).

Suppose now that $C$ is as in (ii).  Set $\L:=\nu^*\O_C(4)\otimes \O_{C^\nu}(-2p_1-p_2)$. 
Since $g_{C^\nu}=2$,  we have $h^0(C^\nu,\L)=12$, by Riemann-Roch. As before, to conclude the proof of (ii) 
it is enough that   there is an injective map
 $H^0(C^\nu,\L)\hookrightarrow T_{[C]} \V_L$.

Indeed, let $W\subset H^0(C,\O_C(4))$ be the subspace obtained as  the image via the restriction map 
 $\rho\:H^0(\Ps^2,\O_{\Ps^2}(4))\ra H^0(C,\O_C(4))$ of the set of sections of $\O_{\Ps^2}(4)$ induced by polynomials $\ol{G}$ satisfying (\ref{Gbar}).  Arguing as for (i), we have an injective map  $W\hookrightarrow T_{[C]}\V_L$.  If $W'\subset H^0(C, \O_C(4))$ is the subspace of sections vanishing on $p$, then $W\subset W'$.  
Let $\nu^*\: H^0(C,\O_C(4))\ra H^0(C^\nu,\nu^*\O_C(4))$ be  the injective map   induced by pullback of sections.   Since $p$ is a node,    any section of $H^0(C^\nu, \nu^*\O_C(4)\otimes\O_{C^\nu}(-p_1-p_2))$ descends to a section of $\O_C(4)$ vanishing on $p$. Thus the restricted map 
$\nu^*|_{W'}\: W' \ra H^0(C^\nu,\nu^*\O_C(4)\otimes\O_{C^\nu}(-p_1-p_2))$ is  an isomorphism,
identifying  $W$ with 
$H^0(C^\nu,\L)$,  by the definition of $W$. Hence  there is an injective map
 $H^0(C^\nu,\L)\hookrightarrow T_{[C]} \V_L$, as required.

To conclude the proof, suppose that $C$ is as in (iii), with $C=Z(F)$, $F\in \co[x,y,z]_4$. Recall that $L=Z(z)$ and $p=[1,0,0]$, hence $F(x,y,0)=0$.   Let $Z(H)$ be a  plane quartic curve such that 
$L\not\subset Z(H)$ and $[Z(H)]\in \V_{L,p}$. Define 
$$\C:=\{(q,t)\in \Ps^2\times \mathbb{A}^1_t \text{ such that }F_t(q):=t\cdot H(q)+F(q)=0\}\subseteq \Ps^2\times \mathbb{A}^1_t.$$ 
If $f\:\C\ra \mathbb{A}^1_t$ is the restriction of the second projection morphism,  then $f$ is a deformation of $C$ to curves of $\V_{L,p}$, because $F_t(x,y,0)=t\cdot H(x,y,0)$, and hence $[C]\in \V_{L,p}$.   

 The vector space $H^0(C,\O_C(4)\otimes\O_C(-3p))$ is the  image via the restriction map 
  $\rho\:H^0(\O_{\Ps^2}(4))\ra H^0(\O_C(4))$ of the set of sections of $\O_{\Ps^2}(4)$ induced by polynomials 
  $G\in\co[x,y,z]_4$ such that $Z(G)$ and $L$ intersect each other with multiplicity at least 3 in $p$, i.e. satisfying 
\begin{equation}\label{Gbarbis}
G(x,y,0)=y^3H, \text{ for some } H\in \co[x,y]_1.
\end{equation}
On the other hand, a first order deformation $F+\epsilon\cdot G$  of $C$  
is a vector of $T_{[C]}\V_{L,p}$, if $(F+\epsilon\cdot G)(x,y,0)$ has at least a triple root in $p$. Then, since  $(F+\epsilon\cdot G)(x,y,0) = \epsilon G(x,y,0)$, the vector space $T_{[C]}\V_{L,p}$ is given by  the first order deformations $F+\epsilon\cdot G$,  for  $G$ satisfying (\ref{Gbarbis}). In this way, $T_{[C]}\V_{L,p}\simeq H^0(C,\O_C(4)\otimes \O_C(-3p))$.
\end{proof}

\begin{Lem}\label{nodal-cubic}
Let $[C_1]\in \V$, where $C_1$ is a nodal plane quartic curve which is  union of an  irreducible cubic $Y$ with a node and a line $L$. Assume that  $Y\cap L$ contains no inflectionary points  of $Y$. If  $\F_{C_1}=\F_{C_2}$, for  some $[C_2]\in \V$, then $C_1=C_2$. In particular, the morphism $\F\:\V\ra Sym^{24}({\Ps^2}^\vee)$   of Proposition \ref{F-map} is generically finite onto its image.
\end{Lem}

\begin{proof} Set $\F:=\F_{C_1}=\F_{C_2}$.
First of all, we will show that $C_2$ is a nodal plane quartic which is the union of an irreducible cubic  with a node and a line.  Indeed, by Remark \ref{ivv},  $\F$ consists of  three lines of type 0 of multiplicity 1, five lines of type 1 of multiplicity 3 and one degenerate line of multiplicity 6. By Proposition \ref{F-conf}(i) $C_2$ does not contain a cusp or a tacnode, i.e. $C_2$ is nodal. Again by Proposition \ref{F-conf}(i), $C_2$ is reducible, otherwise $\F$ would contain an even number of lines of multiplicity 3 or 4. By Proposition \ref{F-conf}(ii),  $C_2$ contains exactly one linear component, hence $C_2$ is a nodal plane quartic which is the union of an irreducible cubic $W$ and a line $M$, and  $W$ has a node, otherwise $\F$ would contain at most  three lines of multiplicity 3.

The linear component both of $C_1$ and of $C_2$ is the reduction of the unique component of multiplicity 6 of $\F$, hence $L=M$.  By contradiction, assume that  $Y\ne W$.  Let  $Y^\vee\subset(\Ps^2)^\vee$ and $W^\vee\subset(\Ps^2)^\vee$ be the dual curves of $Y$ and $W$. Then $Y^\vee\ne W^\vee$.  Since $Y$ and $W$ are irreducible cubics with a node,  $Y^\vee$ and $W^\vee$ have degree 4.   But $Y^\vee$ and $W^\vee$ intersect each other at five smooth points, corresponding to the five components  of multiplicity 3 of $\F$, and at three cusps, corresponding to the three components  of multiplicity 1 of $\F$, contradicting  B\'ezout's theorem.

It follows from the first part of the proof that  
   the cardinality of the fiber of $\F$ over $\F([C_1])$ is one, then by semicontinuity $\F$ is generically finite onto its image.  
 \end{proof}

   For every line $L\subset \Ps^2$ and for every point $p\in L$,  consider the variety $\V_{L,p}$ defined in (\ref{VLp}).   Notice that $\V_{L,p}$ is irreducible,  because it  is the intersection of $\V$ with a linear subspace of $\Ps^{14}$ of codimension 3.   In the proof of the next Theorem,  we will follow the same strategy of~\cite[Theorem 5.2.1]{CS1}.

\begin{Thm}\label{main}
Consider a line $L\subset\Ps^2$ and a point $p\in L$. Consider the morphism $\F\:\V\ra Sym^{24}({\Ps^2}^\vee)$   of Proposition \ref{F-map}. Then $\F|_{\V_{L,p}}$ is injective on a non-empty open subset of $\V_{L,p}$. 
\end{Thm}

\begin{proof}
 Let $C_0$ be a nodal plane quartic curve  which is the union of an irreducible cubic $Y$ with a node and  $L$, with $p\in C_0^{sm}\cap L$, and such that $Y\cap L$ contains no inflectionary points  of $Y$.   Then 
$[C_0]\in \V_{L,p}$, by Lemma \ref{Tang}(iii). Recall the configuration $\F_{C_0}$ in Remark \ref{ivv}.
Set   $\{p_1,p_2,p_3\}:=Y\cap L$ and let  $L_i\subset\F^{red}_{C_0}$ be  the tangent to $Y$  at $p_i$, for $i\in\{1,2,3\}$.  Let $p_4$ be the node of $Y$ and $L_4\subset \F^{red}_{C_0}$ be  one of the two lines of the tangent cone of $C_0$ at  $p_4$.  Let $L_5,L_6,L_7$ be the inflection lines  of type 0 of $C_0$ and $p_i$ be the non-reduced point of $L_i\cap C_0$, for $i\in\{5,6,7\}$.

Let   $\C\ra B$ be a smoothing   of $C_0$ to general curves of $\V_{L,p}$. By contradiction,    assume  the existence of a one parameter family of curves $\W\ra B$ such that $C_b\ne W_b\in \V_{L,p}$,  for $b\in B\setminus \{0\}$, and $\F_{C_b}=\F_{W_b}$, for $b\in B$.  If $b\in B\setminus\{0\}$, then Proposition \ref{F-conf}(iii) implies that  the maximum multiplicity of a component of 
$\F_{C_b}$  is 2, and hence that   $W_b$ is smooth, because $\F_{C_b}=\F_{W_b}$. We have two cases.

\smallskip

\emph{Case 1.}  
Assume that  $[W_0]\in \V$.  Since   $\F_{C_0}=\F_{W_0}$,  by Lemma \ref{nodal-cubic}  we have $C_0=W_0$.  We will show that   $\F|_{\V_{L,p}}$ is an immersion at $C_0$, implying that $C_b=W_b$, for $b\in B\setminus\{0\}$, which is a contradiction.  Let  $\nu\: C^{\nu}_0 \ra C_0$ and $\nu_Y\:Y^{\nu}\ra Y$ be respectively the normalizations of  $C_0$ and $Y$, where $C^{\nu}_0=Y^{\nu}\cup L$.  Set $\{p_{i,1},p_{i,2}\}:=\nu^{-1}(p_i)$, for $i\in\{1,2,3,4\}$, where $p_{i,1}\in Y^{\nu}$, for $i\in\{1,2,3\}$, and let 
$\widehat{\V}$ be the locus of $\V_{L,p}$ parametrizing curves with $L_1,\dots, L_7$ as inflection lines. 
 Then $[C_0]\in \widehat{\V}$ and, up to switching $p_{4,1}$ and $p_{4,2}$, by Lemma~\ref{Tang} we have
\[
T_{[C_0]} \widehat{\V}\subseteq H^0(Y^{\nu}, \L_1)\oplus H^0(L, \L_2),
\]
where 
\[
\L_1:= \nu_Y^*\O_Y(4)\otimes \O_{Y^\nu}\left(-2\left(\sum_{i=1}^4 p_{i,1}\right)-p_{4,2}-2p_5-2p_6-2p_7\right)\simeq\O_{\Ps^1}(-3)
\]
and 
\[
\L_2:=\O_L(4)\otimes \O_L(-p_{1,2}-p_{2,2}-p_{3,2}-3p)\simeq\O_{\Ps^1}(-2).
\]
In this way,  $T_{[C_0]}\widehat{\V} \subseteq H^0(\Ps^1,\O_{\Ps^1}(-3))\oplus H^0(\Ps^1,\O_{\Ps^1}(-2)) =0$. Denoting by $F$ the fiber over the image of $[C_0]$ under the morphism $\F|_{\V_{L,p}}$, that is 
$$F := (\F|_{\V_{L,p}})^{-1}(\F|_{\V_{L,p}}([C_0])),$$ we have the inclusion of tangent spaces 
\[
T_{[C_0]} F \subseteq T_{[C_0]}\widehat{\V}=0,
\]
hence $\F|_{\V_{L,p}}$ is an immersion at $[C_0]$.

\smallskip

\emph{Case 2}. 
Assume that   $[W_0]\notin \V$.  If $W_0$ is GIT-semistable, then either it has an infinite stabilizer, or it is non-reduced, and hence it is a double conic. In any case, either $W_0$ is GIT-unstable, or it has an infinite stabilizer.  By the GIT-semistable replacement property (see \cite[Section 2.1]{CS1}), up to a finite base change totally ramified over $0\in B$,  we can assume that  there are a family $\Z\ra B$ of curves of $\V$ and a morphism $\rho\: B\setminus\{0\}\ra PGL(3)$ such that $Z_b=W_b^{\rho(b)}$ for every $b\in B\setminus\{0\}$. In particular,  $\F_{Z_b}=(\F_{W_b})^{\rho(b)}$, for  $b\in B\setminus\{0\}$.  Recall that  $\F_{C_0}=\F_{W_0}$,  hence $\F_{C_0}$ and  $\F_{Z_0}$ are limits of $PGL(3)$-conjugate families of configuration of lines 
of $\Sym({\Ps^2}^\vee)$. In this way, if  $Orb_{\F_{C_0}}$ is the $PGL(3)$-orbit of $\F_{C_0}$, and  $\ol{Orb}_{\F_{C_0}}$ is its closure in 
 $\Sym({\Ps^2}^\vee)$, then  $\F_{Z_0}\in \ol{Orb}_{\F_{C_0}}$.

\smallskip
 If   $\F_{Z_0}\in Orb_{\F_{C_0}}$, then $\F_{Z_0}=(\F_{C_0})^g=\F_{C_0^g}$, for some $g\in PGL(3)$,  
 and hence $Z_0=C_0^g$, by Lemma \ref{nodal-cubic}. In this way,  we can assume that a family which is a semistable replacement of $\W\ra B$ has $C_0$ as central fiber
and, after acting with $PGL(3)$, we can argue as in Case 1, obtaining a contradiction.

\smallskip 
If $\F_{Z_0}\in \ol{Orb}_{\F_{C_0}}-Orb_{\F_{C_0}}$, then $\F_{Z_0}$ is GIT-unstable, because $\F_{C_0}$ is GIT-stable, by Proposition~\ref{F-stab}(ii). Thus, $Z_0$ contains  either two cusps or one tacnode,  by  Proposition~\ref{F-stab}(i).
 Notice that $\F_{Z_0}$ is degeneration of configurations of lines conjugate to $\F_{C_0}$. 
 
If $Z_0$  contains two cusps, then by Proposition \ref{F-conf}(i) there are  two distinct component $M_1,M_2\subset \F_{Z_0}$ such that $\mu_{M_1}\F_{Z_0}=\mu_{M_2}\F_{Z_0}=8$. Then both $M_1$ and $M_2$ are degenerations of a set of inflection lines  whose multiplicities sum up to 8.  This is not possible because, 
by Remark \ref{ivv}, $\F_{C_0}$ consists of  three lines of multiplicity 1,  five  lines of multiplicity 3 and one line of multiplicity 6.  

If $Z_0$ contains one tacnode, then arguing as in \cite[Proposition 5.1.1, page 241]{CS1},  we can assume  that  $Z_0$  is either irreducible with $Z_0^{sing}$ consisting exactly of one tacnode, or $Z_0$ is the union of a line and a smooth irreducible cubic that  are tangent at a non-inflectionary point of the cubic.   By Remark \ref{ivv}, $Z_0$ has at least six inflection lines of type 0 each  one of which has  multiplicity at most 2.  This is not possible, because $\F_{C_0}$ contains only three lines of multiplicity at most two.
  \end{proof}

\section{Inflection lines of plane cubics}

Let $C \subset \mathbb{P}^2$ be a smooth plane cubic over a field $k$ and let $T_C \subset (\mathbb{P}^2)^\vee$ be the locus of inflection lines to the cubic $C$.  The purpose of this section is to show that the cubic $C$ can be reconstructed from the subset $T_C \subset (\mathbb{P}^2)^\vee$ if the characteristic of $k$ is different from three (Theorem~\ref{n3}).  Note that we need not assume that the field $k$ is algebraically closed.

Inflection points to a smooth plane cubic correspond to cubic roots of the class of a line and are therefore a torsor under the three-torsion subgroup of the Jacobian of the curve.  Thus, in the case of fields of characteristic three, there is little hope of reconstructing a smooth plane cubic from its inflection lines.  More specifically, let $k$ be a field of characteristic three; for $\lambda \in k^\times$ let $C_\lambda$ be the smooth projective cubic 
\begin{equation} \label{hesse}
C_\lambda \colon x^3 + y^3 + z^3 - \lambda xyz = 0 .
\end{equation}
Let $F_\lambda$ be the set of points $\{ [1,-1,0] , [1,0,-1] , [0,1,-1] \}$ of $C_\lambda$.  It is easy to check that $F_\lambda$ is the set of inflection points of the curve $C_\lambda$ and that the inflection lines to $C_\lambda$ are the lines $\{ x = 0 , \,  y = 0 , \, z = 0\}$.  We conclude that the inflection points and the inflection lines for all the cubics $C_\lambda$ coincide, so that in this case it is impossible to reconstruct a smooth plane cubic from the set of inflection lines, even if all the inflection points are given.  For this reason, we assume that the characteristic of the ground field is not three.

The reconstruction method that we follow proceeds essentially by identifying the dual curve to the cubic $C$.  If the characteristic of $k$ is different from two, we show that there is a unique reduced sextic in $(\mathbb{P}^2)^\vee$ having cusps at the nine points in $T_C$.  In the characteristic two case, there are two possibilities.  If the $j$-invariant of the curve $C$ is non-zero, then there is a unique plane cubic in $(\mathbb{P}^2)^\vee$ containing $T_C$ and this cubic is isomorphic to $C$.  If the $j$-invariant of the curve $C$ is zero, then there is a pencil of plane cubics containing $T_C$, but there is a unique cubic in the pencil with vanishing $j$-invariant; this curve is isomorphic to $C$.

To solve the reconstruction problem we study linear systems associated to the anticanonical divisor of the blow up $S_C$ of $(\mathbb{P}^2)^\vee$ along $T_C$.  With the unique exception $({\rm char}(k) , j(C)) = (2,0)$, we show that the anticanonical class on $S_C$ is linearly equivalent to a unique effective divisor and that twice the anticanonical linear system is a base-point free pencil (Corollary~\ref{sepe}).  This implies that the points of $T_C$ are the base locus of a {\em Halphen pencil of index two}.  In the exceptional case the anticanonical linear system itself is a base-point free pencil.

We start by proving a lemma that allows us to compute dimensions of certain special kinds of linear systems on surfaces.

\begin{Lem} \label{ridorido}
Let $X$ be a smooth projective surface over a field and let $N$ be an effective nef divisor on $X$ such that the equality $N^2=0$ holds.  If the divisor $N$ is connected and reduced, then the linear system $|N|$ has dimension at most one, and it is base-point free if it has dimension one.  If the linear system $|N|$ is base-point free, then no element of $|2N|$ is connected and reduced.
\end{Lem}

\begin{proof}
Assume that $N$ is connected and reduced; we first show that the linear system $|N|$ either has dimension zero or it is base-point free.  Write $N=M+F$, where $F$ is the fixed divisor of $|N|$ and $M$ is a base-component free divisor.  By the assumptions, the equalities 
\[
0 = N^2 = (M+F)^2 = M \cdot (M+F) + F (M+F)
\]
hold; since $M+F$ is nef and $M$ and $F$ are both effective, we deduce that the last two summands in the previous equality are non-negative so that also the equalities $M \cdot (M+F) = F \cdot (M+F) = 0$ hold.  Similarly, the divisor $M$ is nef, because it has no base components, and hence $M \cdot M = M \cdot F = 0$.  Since the divisor $N$ is reduced, the divisors $M$ and $F$ have no components in common, and since $N$ is connected, the equality $M \cdot F = 0$ implies that either $M$ or $F$ vanishes.  If $M=0$, then the linear system $|N|$ has dimension zero and we are done.  If $N=0$, then the linear system $|N|$ is base-component free, and the number of base-points is bounded above by $N^2=0$, so that $|N|$ is base-point free.

Thus we reduce to the case in which $N$ is non-zero and the linear system $|N|$ is base-point free, and we need to show that the dimension of $|N|$ is one.  Let $\varphi \colon X \to |N|^\vee$ be the morphism determined by $N$.  Since $N^2=0$, the image $R$ of $\varphi$ is a curve.  The divisor $N$ is the inverse image of a hyperplane $H$ in $|N|^\vee$ under the morphism $\varphi$, and since by assumption $N$ is connected and reduced, we deduce that the intersection of the hyperplane $H$ with $R$ is also reduced and connected.  Since $R$ is a curve, it follows that $R$ has degree one and it is therefore a line.  Because $R \subset |N|^\vee$ is not contained in any hyperplane, it follows that $R=|N|^\vee$, and we conclude that the dimension of $|N|$ is one, as required.

To prove the last assertion, note that the inequality $\dim |N| \geq 1$ implies the inequality $\dim |2N| \geq 2$, so that the linear system $|2N|$ cannot contain a connected and reduced divisor by the first part of the lemma.
\end{proof}

The following result shows that the set of inflection lines to a smooth plane cubic forms the base point of a Halphen pencil of index two with a unique exception; we  analyze separately the exception.  We chose a quick explicit argument to handle the case of characteristic two.  From a characteristic-free point of view, the case of vanishing $j$-invariant is special because it is the only case in which three inflection lines to a smooth plane cubic are concurrent.  This fact, combined with the inseparability of the Gauss map in characteristic two, accounts for the exception.

\begin{Lem} \label{tu}
Let $k$ be a field of characteristic different from three.  Let $C \subset \mathbb{P}^2$ be a smooth cubic curve over $k$, and let $S_C$ be the blow up of the projective plane $(\mathbb{P}^2)^\vee$ dual to $\mathbb{P}^2$ at the points corresponding to the inflection lines to $C$.  The anticanonical divisor ${-K_{S_C}}$ of $S_C$ is nef and linearly equivalent to an effective divisor; the linear system $|{-2K_{S_C}}|$ is base-point free.  Moreover, extending the ground field if necessary, the linear system $|{-2K_{S_C}}|$ contains a connected and reduced divisor, unless ${\rm char}(k) = 2$ and the $j$-invariant of $C$ vanishes.
\end{Lem}

\begin{proof}
Let $T_C \subset (\mathbb{P}^2)^\vee$ denote the set of inflection lines to the cubic $C$, and recall that $T_C$ consists of nine distinct points.  The surface $S_C$ is the blow up of the projective plane $(\mathbb{P}^2)^\vee$ at the nine points in $T_C$.  Since there is always a plane cubic curve containing any nine points, we deduce that the anticanonical divisor of $S_C$ is linearly equivalent to an effective divisor, and that the linear system $|{-2K_{S_C}}|$ contains non-reduced divisors corresponding to twice the divisors in $|{-K_{S_C}}|$.

Let $C^\vee$ denote the strict transform in $S_C$ of the curve dual to $C$.  The curve $C^\vee$ is an irreducible element of the linear system ${-2K_S}$ on $S_C$ and it is reduced if and only if ${\rm char}(k) \neq 2$.  Since the equality $(C^\vee)^2 = 0$ holds, the anticanonical divisor ${-K_{S_C}}$ is nef.  To conclude we show that the linear system $|{-2K_{S_C}}|$ contains a connected and reduced divisor, possibly after an extension of the ground field; the result then follows applying Lemma~\ref{ridorido}.

If the characteristic of the field $k$ is different from two, then the linear system $|{-2K_{S_C}}|$ contains the integral divisor $C^\vee$ and we are done.

Suppose that the $j$-invariant of $C$ is non-zero.  We show, assuming that the ground field is algebraically closed, that there is a connected reduced divisor in the linear system $|{-2K_{S_C}}|$ not containing the reduction $(C^\vee)_{red}$ as a component.  Let 
\[
\mathcal{H} \colon \bigl\{ \lambda_0 (x^3+y^3+z^3) - 3 \lambda_1 xyz = 0 \bigr\}
\]
be the pencil of plane cubics parameterized by $[\lambda_0 , \lambda_1] \in \mathbb{P}^1$; the pencil $\mathcal{H}$ is also knowledge as the {\em Hesse pencil}.  Over an algebraically closed field of characteristic different from three, any isomorphism between plane cubics is induced, up to composition with a translation, by an isomorphism sending an inflection point to an inflection point, and is therefore realized by a projective equivalence.  Thus, over an algebraically closed field of characteristic different from three, every non-singular plane cubic is projectively equivalent to a curve appearing in the Hesse pencil, since the $j$-invariant in the family $\mathcal{H}$ is non-constant.  To prove the result we shall assume that the curve $C$ is the curve 
\[
C_\lambda \colon \bigl\{ x^3+y^3+z^3 - 3 \lambda xyz = 0 \bigr\}
\]
in the pencil $\mathcal{H}$ corresponding to the parameter $[1,\lambda] \in \mathbb{P}^1$.  The inflection points of the curve $C_\lambda$ are the nine points 
\[
[1 , \varepsilon , 0] \quad , \quad [1 , 0 , \varepsilon ] \quad , \quad [0 , 1 , \varepsilon], \quad {\textrm{where }} \varepsilon ^3 + 1 = 0 ,
\]
and the corresponding inflection lines are the lines with equations 
\[
x + \omega y + \lambda \omega^2 z = 0 \quad , \quad x + \lambda \omega^2 y + \omega z = 0 \quad , \quad \lambda \omega^2 x + y + \omega z = 0, \quad {\textrm{where }} \omega ^3   = 1 .
\]
It is therefore clear that the sextic $(\lambda x^2 - yz) (\lambda y^2 - xz) (\lambda z^2 - xy)$ corresponds to an element $\Sigma_\lambda$ of $|{-2K_{S_C}}|$ that is connected and reduced, if $\lambda \neq 0$.  Since the divisor $\Sigma_\lambda$ is a union of conics, it follows that $\Sigma_\lambda$ does not contain $C^\vee$.  To conclude, observe that the $j$-invariant of the curve $C_0$ vanishes.
\end{proof}

\begin{Rem}
The exception mentioned in Lemma~\ref{tu} is necessary.  Let $F \subset \mathbb{P}^2$ be the Fermat cubic curve with equation $x^3+y^3+z^3=0$ over a field of characteristic different from three.  As in the proof of Lemma~\ref{tu}, 
the inflection points of the curve $F$ are the points with coordinates 
\begin{equation} \label{coorte}
[1 , \epsilon , 0] \quad , \quad [1 , 0 , \epsilon] \quad , \quad [0 , 1 , \epsilon], \quad {\textrm{where }} \epsilon ^3+1  = 0 
\end{equation}
and the corresponding inflection lines are the points with coordinates $[1,-\varepsilon^2,0]$ in $(\Ps^2)^\vee$ (up to permutations).  Therefore, the coordinates of the inflection points of the Fermat cubic involve cubic roots of $-1$, while the coordinates of its inflection lines involve cubic roots of $1$: in the case of a field of characteristic two, the configuration of inflection points and the configuration of inflection lines are projectively equivalent!  Thus we see that the set $T$ of the points in~\eqref{coorte} is the base locus of the Hesse pencil $\mathcal{H}$, and we conclude that the anticanonical linear system of the surface $S_F$ obtained by blowing up $\mathbb{P}^2$ along $T$ has dimension at least one.  It follows at once that the dimension of the anticanonical linear system on $S_F$ is one.
\end{Rem}

\begin{Cor} \label{sepe}
Let $k$ be a field of characteristic different from three.  Let $C \subset \mathbb{P}^2$ be a smooth cubic curve over $k$, and let $S_C$ be the blow up of the projective plane $(\mathbb{P}^2)^\vee$ dual to $\mathbb{P}^2$ at the points corresponding to the inflection lines to $C$; if the characteristic of $k$ is two, then assume also that the $j$-invariant of $C$ is non-zero.  The linear system $|{-2K_{S_C}}|$ on $S_C$ is a base-point free pencil and the associated morphism $S_C \to \mathbb{P}^1$ is a rational elliptic surface with a unique multiple fiber corresponding to twice the unique element of $|{-K_{S_C}}|$.
\end{Cor}

\begin{proof}
By Lemma~\ref{tu}, the linear system $|{-2K_{S_C}}|$ is base-point free and contains a connected and reduced divisor $R$.  Applying Lemma~\ref{ridorido} with $X=S_C$ and $N=R$, we deduce that the dimension of the linear system $|{-2K_{S_C}}|$ is one and that the dimension of the linear system $|{-K_{S_C}}|$ is zero.  The result follows.
\end{proof}

\begin{Thm} \label{n23}
Let $C \subset \mathbb{P}^2$ be a smooth plane cubic curve over a field $k$ of characteristic relatively prime to six and let $T_C \subset (\mathbb{P}^2)^\vee$ be the set of inflection lines of $C$.  If a plane sextic $(\mathbb{P}^2)^\vee$ is reduced, singular at all the points in $T_C$ and has a cusp at one of the points in $T_C$, then it is the dual curve $C^\vee$ of $C$ and in particular all of its singular points are cusps.
\end{Thm}

\begin{proof}
As before, denote by $S_C$ the blow up of $(\mathbb{P}^2)^\vee$ at the points of $T_C$, so that $S_C$ is a smooth projective rational surface.  By Corollary~\ref{sepe}, the linear system $|{-2K_{S_C}}|$ induces a morphism $\varphi \colon S_C \to \mathbb{P}^1$ exhibiting $S_C$ as a rational elliptic surface with a multiple fiber $B$.  One of the fibers of $\varphi$ corresponds to the sextic curve $C^\vee \subset (\mathbb{P}^2)^\vee$ dual to the curve $C$: we denote this fiber of $\varphi$ by $\overline{C}^\vee$.

Let $f \in T_C$ be a point corresponding to an inflection line to $C$.  Denote by $E_f$ the exceptional curve of $S_C$ lying above the point $f$, so that $(E_f)^2 = K_{S_C} \cdot E_f = -1$.  A reduced curve $D$ in $|{-2K_{S_C}}|$ intersects the exceptional curve $E_f$ at a subscheme of length two that is non-reduced exactly when the plane sextic corresponding to $D$ has a cusp at $f$.  Moreover, the intersection of $E_f$ with a fiber of $\varphi$ is non-reduced precisely when the morphism $\varphi |_{E_f}$ ramifies.  The restriction of the morphism $\varphi$ to the curve $E_f$ has degree $(-2K_{S_C}) \cdot E_f = 2$, and hence it has exactly two ramification points.  One of the ramification points is $E_f \cap B$, since the intersection of $E_f$ with the multiple fiber is non-reduced.  We also know that the morphism $\varphi |_{E_f}$ ramifies at the point $E_f \cap \overline{C}^\vee$, since the sextic corresponding to $\overline{C}^\vee$ is the dual of $C$ and hence it has a cusp at $f$.  We therefore see that the two ramification points of the morphism $\varphi |_{E_f}$ are one in the multiple fiber $2B$ and the other in the fiber $\overline{C}^\vee$.  Since this is true for all points $f \in T_C$, we conclude that the only reduced plane sextic having a singular point at each point of $T_C$ one of which is a cusp is the sextic $C^\vee$, and we are done.
\end{proof}

\begin{Thm}[Reconstruction Theorem] \label{n3}
Let $C \subset \mathbb{P}^2$ be a smooth plane cubic curve over a field $k$ of characteristic different from three and let $T_C \subset (\mathbb{P}^2)^\vee$ be the set of inflection lines of $C$.  There is a unique (geometrically) integral curve $C' \subset (\mathbb{P}^2)^\vee$ such that 
\begin{itemize}
\item
if ${\rm char}(k) \neq 2$, then $C'$ is a sextic with cusps at the points of $T_C$;
\item
if ${\rm char}(k) =2$ and $j(C) \neq 0$, then $C'$ is a cubic containing $T_C$;
\item
if ${\rm char}(k) =2$ and $j(C) = 0$, then $C'$ is a cubic containing $T_C$, with vanishing $j$-invariant.
\end{itemize}
Moreover, the space of cubics in $(\mathbb{P}^2)^\vee$ containing $T_C$ has dimension one if and only if $({\rm char}(k),j(C)) = (2,0)$.  In all cases, the curve $C'$ described above is the dual of the curve $C$.
\end{Thm}

\begin{proof}
If the characteristic of the field $k$ is different from two, then the result is a consequence of Theorem~\ref{n23}.

Suppose now that the characteristic of the field $k$ is two.  In this case, the Gauss map is purely inseparable and the reduced image $C^\vee$ of the Gauss map of the curve $C$ is a cubic in $(\mathbb{P}^2)^\vee$ containing $T_C$.  If the $j$-invariant of the curve $C$ is non-zero, then Corollary~\ref{sepe} implies that there is a unique cubic $C'$ in $(\mathbb{P}^2)^\vee$ containing $T_C$, and hence the cubic $C'$ must be the curve $C^\vee$, as required.

Finally, suppose that the $j$-invariant of the curve $C$ is zero (and ${\rm char}(k) = 2$), so that the linear system $\Lambda$ of cubics in $(\mathbb{P}^2)^\vee$ containing $T_C$ has dimension at least one.  The reduced dual $C^\vee$ of $C$ is an integral cubic in $(\mathbb{P}^2)^\vee$ containing $T_C$; Lemma~\ref{ridorido} shows that the linear system $\Lambda$ has dimension exactly one.  Since $\Lambda$ is a pencil of plane cubics containing smooth fibers, we deduce that the $j$-invariant of the pencil is a morphism of degree twelve.  Since the curve $C$ has $j$-invariant zero and the characteristic of the field $k$ is two, the automorphism group of the curve $C$ has order 24 over any algebraically closed extension of $k$.  In particular, the multiplicity of the fiber of the $j$-invariant corresponding to the curve $C^\vee$ is twelve, and hence there is a unique curve with $j$-invariant zero in the pencil $\Lambda$, namely $C^\vee$.  Thus also in this case, the curve $C'$ coincides with the curve $C^\vee$, and the result follows.
\end{proof}

\section*{Acknowledgments} 
\noindent We want to thank  E. Esteves  for  fundamental suggestions,  and also  A. Abreu, L. Caporaso,  L. Gatto, D. Lehavi,  N. Medeiros,  E. Sernesi for precious discussions and 
 I. Dolgachev for pointing out the reference \cite{W}. We also thank the anonymous referee for his detailed report and for the constructive comments.

\end{document}